\documentclass[12pt,reqno]{amsart}
\usepackage{xifthen,csquotes}
\usepackage{esvect}
\usepackage{bm}
\usepackage{extarrows}
\usepackage{multirow}
\usepackage{subfig}
\usepackage{pkgfile}

\begin{document}

\title[Signal Detection in Degree corrected ERGMs]{Signal Detection in Degree corrected ERGMs}

\author[Y.~Xu \& S.~Mukherjee]{Yuanzhe Xu and Sumit Mukherjee}

\address{Department of Statistics, Columbia University
\newline\indent 1255 Amsterdam Avenue, New York, NY 10027}

\date{\today}




\maketitle

\begin{abstract}
In this paper, we study sparse signal detection problems in \enquote{degree corrected} Exponential Random Graph Models (ERGMs). We study the performance of two tests based on the conditionally centered sum of degrees and conditionally centered maximum of degrees, for a wide class of such ERGMs. The performance of these tests match the performance of the corresponding uncentered tests in the ${\bm \beta}$ model (\cite{mukherjee2018detection}). Focusing on the degree corrected two star ERGM, we show that improved detection is possible at \enquote{criticality} using a test based on (unconditional) sum of degrees. In this setting we provide matching lower bounds in all parameter regimes, which is based on correlations estimates between degrees under the alternative, and of possible independent interest.
%
\end{abstract}
\section{Introduction}

Studying network models has a long and rich history in Statistics, with applications across various disciplines such as Social Science, Biology, Neuroscience, Climatology, and Ecology, to name a few. One of the most well known network models is the Exponential Random Graph Model (often abbreviated as ERGM). ERGMs originated in the Social Science Literature 
(c.f.~\cite{anderson1999p,frank1986markov,holland1981exponential, robins2007introduction,wasserman1994social,wasserman1996logit} and the references there-in), 
and have since then received considerable attention in Statistics and Probability (c.f.~\cite{ chatterjee2013estimating, chatterjee2011random, gotze2021concentration, mukherjee2018detection, mukherjee2018global,schweinberger2020concentration,shalizi2013consistency} and references there-in). ERGMs represent exponential families of distributions the space of simple labeled graphs with a finite dimensional sufficient statistics, which are 
usually taken to be subgraph counts. The simplest class of examples under this framework consists of the one parameter ERGMs, which admits a one dimensional sufficient statistic. Below we start by introducing such a one parameter ERGM:

 Letting $\mathcal{G}_n$ denote the set of all simple labeled graphs $G$ with vertex set $[n]:=\{1,2,...,n\}$, we 
consider the following probability mass function on $\mathcal{G}_n$:
\begin{align}\label{eq:ergm}
\P_{n,{ \theta}}(G):=\frac{1}{Z_n(\theta,H)}\exp\Big\{\theta  \frac{N(H,G)}{n^{\zeta-2}}\Big\}.
\end{align}
Here 
\begin{enumerate}
\item[(i)]
$H$ is a graph of fixed size (such as an edge, triangle, cycle, star, etc.), 

\item[(ii)] $N(H,G)$ is the number of copies of the graph $H$ in the graph $G$, 

\item[(iii)] $\zeta$ is the number of vertices in the graph $H$,

\item[(iv)] $\theta$ is a real valued parameter,

\item[(v)] $Z_n(\theta,H)$ is the normalizing constant.

\end{enumerate}
In particular if the graph $H$ is an edge, the model in \eqref{eq:ergm} is an Erd\H{o}s-R\'enyi model, where the edges of the graph $G$ are i.i.d from suitable a Bernoulli distribution.~For any other choice of $H$, the model in \eqref{eq:ergm} is not an Erd\H{o}s-R\'enyi model since one allows nontrivial dependence between the edges.  An ERGM can thus be thought of as a natural generalization of the Erd\H{o}s-R\'enyi model, which allows for growing degrees of  dependence between edges by through the indexing subraph $H$. It is natural to allow for this dependence while modeling networks, to incorporate features like \enquote{friends of friends are more likely to be friends}.
However, one drawback of ERGMs (or at least the model introduced in \eqref{eq:ergm}) is that the edges of the random graph are still jointly exchangeable, in the sense that permuting the vertices of $G$ does not change the distribution of the graph $G$. Consequently the degree sequence $(d_1(G),\ldots,d_n(G))$ \footnote{$d_i(G)=\sum_{j}G_{ij}$ with $\{G_{ij}\}_{i,j\in [n]}$ being the adjacency matrix of $G$.} marginally have the same distribution for each $i\in [n]$. This may not be desirable for modeling networks where there are a few vertices of very high degree ( (see \cite{bhamidi2015twitter}), compared to the remaining vertices. Such a feature is often present in social networks, where the vertex corresponding to a popular/famous person has a very high degree compared to the remaining vertices.

 One model which captures degree homogeneity is the $\beta$-model of social networks (c.f.~\cite{blitzstein2011sequential,chatterjee2019estimation,chatterjee2011random,mukherjee2018detection,rinaldo2013maximum} and references there-in). The $\beta$-model is defined by the following p.m.f.~on $\mathcal{G}_n$:
\begin{align}\label{eq:beta}
\P_{n,{\bm \beta}}(G):=\frac{1}{Z_n({\bm \beta})}\exp\Big\{\sum_{i=1}^n\beta_i d_i(G)\Big\}.
\end{align}
Here 
\begin{enumerate}
\item[(i)]
$(d_1(G),\ldots,d_n(G))$ is the degree sequence of the graph $G$.

\item[(ii)] ${\bm \beta}=(\beta_1,\ldots,\beta_n)^T\in\mathbb{R}^n$ is a vector valued parameter,

\item[(iii)] $Z_n({\bm \beta})$ is the normalizing constant.

\end{enumerate}

In this model, for each vertex $i\in [n]$ there is a real valued parameter $\beta_i$ which controls the effect of the $i^{th}$ vertex, and consequently the typical size of the degree $d_i(G)$. This allows for heterogeneity among the degrees. A large value of $\beta_i$ results in a large value of the degree of the $i^{th}$ vertex, and vice versa. One drawback of the ${ \beta}$-model \eqref{eq:beta} is that the edges of the graph $G$ are no longer dependent. This is not immediate from \eqref{eq:beta}, but is not hard to check (see for e.g.~\cite{chatterjee2011random}). Thus although the $\beta$-model allows for degree heterogeneity, it does not involve dependence between the edges.

A natural way to retain both the dependence between edges and the heterogeneity of the degrees is to consider an exponential family which has both the terms $\theta N(H,G)$ and $\sum_{i=1}^n\beta_id_i(G)$ in the exponent. Indeed, dependence between edges is present because of the term $\theta N(H,G)$, and degree heterogeneity is present because of the term $\sum_{i=1}^n\beta_id_i(G)$. Such a model, which we introduce formally below, can be thought of as a degree corrected ERGM. 


\subsection{Degree corrected ERGM} 
As before, let $\mathcal{G}_n$ denote the set of all simple labelled graphs $G$ with vertex set $[n]:=\{1,2,...,n\}$,
Given a graph $G\in \mathcal{G}_n$, by slight abuse of notation we use $G$ to also denote the adjacency matrix of $G$, defined as follows:
\begin{align}
G_{ij}=\left\{
\begin{array}{rcl}
1&     &\text{If an edge is present  between vertices $i$ and $j$ in $G$,}\\
0&     &  \text{If no edge is present  between vertices $i$ and $j$ in $G$}.
\end{array} \right.
\end{align} 
Thus, we encode presence or absence of edges by $\{0,1\}$. By convention, set $G_{ii}:=0$, and note that $G$ is a symmetric $n\times n$ matrix with $0$ on the diagonal, and $\{0,1\}$ entries on the off-diagonals.
Let $(d_{1},d_{2},\cdots,d_n)$ denote the labeled degree sequence of the graph $G$, defined by
$$d_{i}:=\sum_{j=1}^{n}G_{ij}, 1\le i\le n.$$
Let $H$ be a fixed connected subgraph with $\zeta:=|V(H)|\ge 2$ (i.e. $H$ is not an isolated vertex). Assume that the vertices of $H$ are labeled as $[\zeta]=\{1,2,\ldots, \zeta\}$. Let $\mathcal{I}_{n}$ denote the 
 the set of all 1-1 maps  from $[n]$ to $[\zeta]$.
 For any $G\in \mathcal{G}_n$, let $N(H,G)$ denote the number of copies of $H$ in $G_n$, defined by 
 $$N(H,G)=\sum\limits_{\iota\in\mathcal{I}_{n}}\prod_{(i,j)\in E(H)}G_{\iota(i),\iota(j)},$$
 where $E(H):=\{(a,b)\in V(H):(a,b)\text{ is an edge in }H\}$ is the edge set of $H$.
 As for illustration, the expression of $N(H,G_n)$ when $H$ is an edge, triangle, and two star (to be denoted by $K_2, K_3, K_{1,2}$ respectively) are respectively given by:
 \begin{align*}
 N(K_2,G)=&\sum_{i\ne j}G_{ij}=2\sum_{i<j}G_{ij}=\sum_{i=1}^nd_i,\\
 N(K_3,G)=&\sum_{i\ne j\ne k}G_{ij} G_{jk} G_{ki}=6 \sum_{i<j<k} G_{ij} G_{jk} G_{ki},\\
 N(K_{1,2},G)=&\sum_{i\ne j\ne k}G_{ij} G_{ik}=2\sum_{i=1}^n \sum_{j<k} G_{ij} G_{ik}=2\sum_{i=1}^n{d_i\choose 2}.
 \end{align*}
 Given a parameter $\theta>0$ and vector $\bm{\beta}=(\beta_{1},\beta_{2},...,\beta_{n})\in \R^n$, we subsequently define a probability mass function on $\mathcal{G}_{n}$ by setting
\begin{align}\label{eq:ERGM}
\P_{n,\theta,{\bm \beta}}(G):=\frac{1}{{Z}_n(\bm{\beta},\theta,H)} \text{exp}\Big\{\frac{\theta}{n^{\zeta-2}}N(H,G)+\sum\limits_{i=1}^{n}\beta_{i}d_i\Big\}.
\end{align}
where as usual ${Z}_n(\bm{\beta},\theta,H)$ is the normalizing constant. The scaling $n^{\zeta-2}$ ensures that the resulting model is non-trivial as $n\to\infty$ (c.f.~\cite{chatterjee2013estimating}). If $\beta_i=\beta_0$ for some $\beta_0\in \R$ free of $i$, then the model in \eqref{eq:ERGM} is an Exponential Random Graph Model with two sufficient statistics $N(H,G)$ and $E(G)$, where $E(G)=\frac{1}{2}N(K_2,G)$ is the number of edges in the graph $G$. 
In this case the random graph $G$ represents a bivariate exchangeable array. More precisely, for any permutation $\pi\in S_n$ the graph $G_\pi$ defined by $G_\pi(i,j):=G_{\pi(i),\pi(j)}$ has the same distribution as $G$, i.e.~$G_\pi\stackrel{D}{=}G$. The vector of parameters $\bm{\beta}$,  therefore,  measures the individual effects of each vertex, and for a general vector $\bm{\beta}$ a random graph $G$ from the model \eqref{eq:ERGM} is no longer exchangeable. 
For $\theta>0$, the term $N(H,G)$ ensures that there is positive dependence among the edges in $G$, in the sense that conditional on presence of an edge, any other edge is more likely to be present.  If $\theta=0$, the model \eqref{eq:ERGM} reduces to the $\beta$-model as in \eqref{eq:beta}, in which all edges $G_{ij}$ are independent, with $$\P_{n,0,\bm{\beta}}(G_{ij}=1)=\frac{e^{\beta_i+\beta_j}}{1+e^{\beta_i+\beta_j}}.$$
Thus the model in \eqref{eq:ERGM} combines the features of the $\beta$-model and traditional ERGMs. We will use the term degree corrected ERGM to refer to the model \eqref{eq:ERGM}.

\subsection{Hypothesis testing problem for  \texorpdfstring{$\bm{\beta}$}{1}}

Given the model \eqref{eq:ERGM}, a natural question is to carry out inference regarding the vector ${\bm \beta}$. In the setting where $\theta=0$, the problem of estimation of ${\bm \beta}$ using the MLE $\hat{\bm \beta}_{ML}$ was studied in \cite{chatterjee2011random}, where the authors gave bounds on $||\hat{\bm \beta}_{ML}-{\bm \beta}||_\infty$. The question of testing of the grand null hypothesis ${\bm \beta}={\bm 0}$ versus non negative sparse alternatives was studied in \cite{mukherjee2018detection}, where the authors show that the optimal test depends on the sparsity level and strength of the signal. Since both these papers assumed $\theta=0$, the edges of the graph $G$ were independent, which was used significantly in the proofs of the results. A natural question is whether one can extend these results in the presence of dependence between edges, i.e. when $\theta>0$. In this paper, we study the question of testing the grand null hypothesis ${\bm \beta}=\beta_0{\bf 1}$ against sparse one sided alternatives. Essentially we want to test the null hypothesis that all nodes in the network are equally popular (have the same $\beta_i$), versus the alternative hypothesis that there is a small hub of nodes which are more popular (have a higher value of $\beta_i$) compared to the baseline popularity $\beta_0$ of the remaining nodes. Here ${\beta_0}\in \R$ is a real valued parameter which is assumed to be known. In section \ref{sec:mc_fs} we briefly discuss what can go wrong if the parameter $\beta_0$ is not assumed to be known.  Below we formally introduce the testing problem discussed above.
\\

Let $\beta_0\in \R$ be known. Let $G$ be a graph drawn from the probability distribution \eqref{eq:ERGM}, and for a known $\theta>0$ and given $\beta_0\in \R$  we consider the following hypothesis testing problem:
\begin{align}\label{Hypothesis_testing}
\mathcal{H}_{0}:\bm{\beta}=\beta_{0}{\bf 1}\text{\ \ \ vs\ \ \ }\mathcal{H}_{1}:\bm{\beta}\in\Xi(s,A).
\end{align}
Here under the null hypothesis we have $\beta_i=\beta_0$ for all $i\in [n]$ and  we denote this null probability measure as $\P_{n,\theta,\beta_{0}}$. The set of vectors $\Xi(s,A)$ in the alternative hypothesis $H_1$ is defined as
\begin{align}\label{Alternative}
\Xi(s,A):=\left\{\bm{\beta}=\beta_0{\bf 1}+{\bm \mu}: |\text{supp}\bm{\mu}|\geq s,\text{ and }\min\limits_{i\in\text{supp}\bm{\mu}}\mu_{i}\geq A\right\}.
\end{align}
In words, under the alternative hypothesis there is a sparse set $S$ of size $s$, such that $\beta_i\ge \beta_0+A$ if $i\in S$, and $\beta_i=\beta_0$ if $i\notin S$.
Our main goal of this paper is to study the effect of the nuisance parameter $\theta$ on the hypothesis testing problem \eqref{Hypothesis_testing}.
For studying the proposed hypothesis testing problem, here we adopt an asymptotic minimax framework similar to \cite{mukherjee2018detection,mukherjee2018global}, which is introduced below (see also \cite{burnavsev1979minimax,ingster1994minimax,ingster1998minimax,ingster2012nonparametric}). 
\\

Given a non randomized test function $T_{n}:\mathcal{G}_n\mapsto \{0,1\}$,
define the risk of test $T_{n}(G)$ as the sum of type I and type II errors, as follows:
\begin{align}\label{Risk}
R(T_{n},\Xi(s,A),\bm{\beta}):=\P_{n,\theta,\beta_{0}}(T_{n}(G)=1)+\sup\limits_{\bm{\beta}\in\Xi(s,A)}\P_{n,\theta,{\bm \beta}}(T_{n}(G)=0).
\end{align}
Given a sequence of test functions $\{T_{n}\}_{n\geq1}$ for the testing problem \eqref{Hypothesis_testing}, 
we call $\{T_{n}\}_{n\geq1}$ as 
\begin{enumerate}
\item[(i)] Asymptotically Powerful, if 
\begin{align}\label{Asymp_Power}
\lim\limits_{n\rightarrow\infty}R(T_{n},\Xi(s,A),\bm{\beta})=0;
\end{align}
\item[(ii)] Asymptotically not Powerful, if
\begin{align}\label{Asymp_Power0}
\liminf\limits_{n\rightarrow\infty}R(T_{n},\Xi(s,A),\bm{\beta})>0;
\end{align}
\item[(iii)] Asymptotically Powerless, if
\begin{align}\label{Asymp_Powerless}
\lim_{n\rightarrow\infty}R(T_{n},\Xi(s,A),\bm{\beta})=1.
\end{align}
\end{enumerate}
By definition,  both type I and type II errors converge to 0 for asymptotically powerful tests. Also, if a sequence of tests is asymptotically powerless, then it is also asymptotically not powerful, and so (iii) is a stronger notion than (ii).

\subsection{Main Results}\label{Results}

In this section we present and discuss our main results. To that end, we first consider general degree corrected ERGMs and analyze the performance of two natural tests. 
 We then focus on a particular degree corrected ERGM, where the graph $H$ is a two star. In this setting we show that the general tests studied above attains the \enquote{optimal detection boundary} for all configurations $(\theta,\beta_0)$ barring a specific point, which we refer to as the critical point/configuration. At this point, using a slightly different test from the ones studied under the general ERGM framework, we are able to detect much lower signals, compared to the independent case $(\theta=0$). 
\subsubsection{General degree corrected ERGMs}
In this section, we discuss the hypothesis testing problem \eqref{Hypothesis_testing} in the setting of general degree corrected ERGMs as in \eqref{eq:ERGM}. Specifically, we will show how signal density and strength $(s,A)$ coordinate to determine the threshold for testing efficiency. Two natural test statistics for this problem are the sum of degrees $\sum_{i=1}^nd_i(G)$, and the maximum degree $\max_{i\in [n]}d_i(G)$. However, because of the presence of dependence, it is very difficult to calibrate the cut-off for these statistics, as they depend on the parameter $\theta$ in a non-trivial way. To counter this, we use conditionally centered versions of the sum of degrees, and the maximum degree, similar to what was done in \cite{mukherjee2018global}. 
\\

Our first theorem studies the performance of a test based on the  conditionally centered sum of degrees. For stating the result we require a few notations.
\begin{defn}
Let $\mathcal{E}:=\{(i,j): 1\le i<j\le n\}$ be the set of all edges in the complete graph $K_n$. For any $e=(i,j)\in \mathcal{E}$, 
 let $N_e(H,G)$  denote the number of copies of $H$ in the graph $G$ which contains the edge $e$, and let $N_{e,f}(H,G)$ denote the number of copies of $H$ in the graph $G$ which contains both the edges $e,f$. 

Setting $\psi(x):= \frac{e^x}{1+e^x}$ for $x\in \R$, for any $e=(i,j)\in \mathcal{E}$ we have
\begin{align}\label{conditional_mean}
\E_{n,\theta,{\bm \beta}}\Big(G_{e}\big|G_{f}:f\neq e\Big)=\psi(\theta t_e(H,G)+\beta_i+\beta_j),
\end{align}
where $t_e(H,G):=\frac{N_e(H,G)}{n^{\zeta-2}}$.
\end{defn}
Since our results are asymptotic in nature, below we introduce some standard notations, to be used in the remainder of the paper. 
\begin{defn}
Given two sequence of real numbers $\{a_n\}_{n\ge 1}$ and $\{b_n\}_{n\ge 1}$, we use the notation $a_n=O(b_n)$ or $a_n\lesssim b_n$ to imply the existence of a positive finite constant $c$ free of $n$, such that $a_n\le c b_n$. We use the notation $a_n\gg b_n$ ($a_n\ll b_n$) to imply $\lim_{n\to\infty} \frac{a_n}{b_n}=\infty$ ($\lim_{n\to\infty}\frac{a_n}{b_n}= 0$ respectively). 
\end{defn}
\begin{thm}\label{Centered_mean_test}
With $G$ from the model \eqref{eq:ERGM}, consider the hypothesis testing problem described in \eqref{Hypothesis_testing}. If $sA\rightarrow\infty$, then for any sequence $L_n$ such that 
$ n\ll L_n\ll nsA$ the conditionally centered sum of degrees test $T_n(G)$ given by
\begin{align*}
T_n(G)=&1\text{ if }\sum\limits_{e\in \mathcal{E}}\bigg[G_{e}-\E_{n,\theta,\bm{\beta}_{0}}\big(G_{e}\big|G_{f}:f\neq e\big)\bigg]>L_{n},\\
=&0\text{ otherwise}
\end{align*}
is asymptotically powerful.
\end{thm}
In settings where the signal size $s$ is small, a test based on the conditionally centered maximum of degrees can sometimes detect lower signals. The performance of this test is studied in our second result.
\begin{thm}\label{Centered_max_test}
With $G$ from the model \eqref{eq:ERGM}, consider the hypothesis testing problem described in \eqref{Hypothesis_testing}. Then there exists constants $\kappa,C$  such that if $A\geq\kappa\sqrt{\frac{\log n}{n}}$ and $L_n= C\sqrt{n\log n}$, then the conditionally centered maximum degree test defined by 
\begin{align}\label{conditional_max_test}
T_{n}(G)=\left\{
\begin{array}{rcl}
1&     &  \text{If\ \ } \max\limits_{i\in[n]}\sum\limits_{e\ni i}\Big[G_{e}-\E_{n,\theta,\bm{\beta}_{0}}\big(G_{e}\big|G_{f}:f\neq e\big)\Big]>L_{n}\\
0&     &  \text{Otherwise}
\end{array} \right.
\end{align}
 is asymptotically powerful.
%
%
%
\end{thm}
Comparing Theorem \ref{Centered_mean_test} and \ref{Centered_max_test} yields that the conditionally centered maximum degree test is better (has a lower detection boundary) for sparser alternative ($s\ll\sqrt{\frac{n}{\log n}}$), and the conditionally centered sum of degrees test is better for denser alternatives ($s\gg\sqrt{\frac{n}{\log n}})$. This is similar to the findings of \cite{mukherjee2018detection}, where it was shown that optimal rate detection is obtained by the sum of degrees if $s=n^{b}$ with $b> 1/2$ (see \cite[Theorem 3.1]{mukherjee2018detection}), and by the maximum degree test if $b<1/2$ (see \cite[Theorem 3.3]{mukherjee2018detection}).

%

\subsubsection{Degree Corrected Two-star ERGM}\label{phi}

In Theorems \ref{Centered_mean_test} and \ref{Centered_max_test}, there is no effect of the nuisance parameter $\theta$ on the detection rate of the tests. To demonstrate that the best possible detection rate can change depending on the value of $\theta$, we study in detail the degree corrected two star ERGM, The two star is the graph $K_{1,2}$, which is a path of length 3.
For notational and computational convenience, for the Degree Corrected Two-star ERGM our edge variables take values in $\{-1,1\}$ instead of $\{0,1\}$. More precisely, given a graph $G\in \mathcal{G}_n$, our adjacency matrix $Y$ is now defined as follows:
\begin{align*}
Y_{ij}=&+1\text{ if }(i,j)\text{ is an edge in }G,\\
=&-1\text{ if }(i,j)\text{ is not an edge in }G.
\end{align*}
As before, we set $Y_{ii}=0$ by convention. Thus $Y$ is a symmetric matrix with $\{-1,1\}$ entries, and $0$ on the diagonal. 
Let $(k_{1},k_{2},\cdots,k_n)$ denote the labeled \enquote{degree sequence} of the graph $Y$, i,e, 
$$k_{i}:=\sum_{j=1}^{n}Y_{ij}, 1\le i\le n.$$ 
The following display introduces the degree corrected two star ERGM as a p.m.f.~on $\{-1,1\}^{n\choose 2}$:
\begin{align}\label{Y:edge_star}
\P_{n,\theta,\bm{\beta}}(Y)=\frac{1}{Z_{n}(\bm{\beta},\theta)}\exp\left\{\frac{\theta}{n-1}\widetilde{N}(K_{1,2},G_n)+\frac{1}{2}\sum\limits_{i=1}^{n}\beta_{i}k_{i}\right\},
\end{align}
where $$\widetilde{N}(K_{1,2},G_n):=\sum_{i=1}^n \sum_{j<k} Y_{ij} Y_{ik}=\frac{1}{2}\sum_{i=1}^nk_i^2-\frac{n(n-1)}{2}.$$
Having observed $Y$, consider the same hypothesis testing problem \eqref{Hypothesis_testing} as above.
For the sake of clarity of presentation, in this section we parametrize the signal size $s$ and signal strength $A$ by $n^b$ and $n^{t}$ respectively, where $b\in(0,1)$ and $t<0$.
The detection boundary for this problem shows a phase transition depending on the nuisance parameter $\theta$. Stating this requires the following partitioning of the parameter space for  $(\theta,\beta_0)$:
\begin{defn}
\noindent
\begin{itemize}
    \item Let $\Theta_1=\Theta_{11}\cup \Theta_{12}$, where $\Theta_{11}:=(0,1/2)\times \{0\}$, and $\Theta_{12}=\{(\theta,\beta_0):\theta>0, \beta_0\ne 0\}.$
    
   \item
    Let $\Theta_2:= (1/2,\infty)\times \{0\}$. 
    \item 
    Let $\Theta_3:=(1/2,0)$ 
   Note that $\Theta_1\cup \Theta_2\cup\Theta_3=(0,\infty)\times \R$.
\end{itemize}
\end{defn}
Our first result describes the detection boundary for the degree corrected two star ERGM if $(\theta,\beta_0)\in\Theta_1$.
\begin{thm}\label{Unique_dense}
Let $Y$ be an observation from from \eqref{Y:edge_star}, and assume $(\theta,\beta_{0})\in \Theta_1$. Consider the hypothesis testing problem  described in \eqref{Hypothesis_testing} with $s=n^{b}$ and $A=n^t$ for $\in (0,1)$ and $t<0$.

\begin{enumerate}[(a)]
\item If $b\ge \frac{1}{2}$ and $b+t<0$, all tests are asymptotically powerless.

\item If $b\ge \frac{1}{2}$ and $b+t>0$, then the conditionally centered sum test of Theorem \ref{Centered_mean_test} is asymptotically powerful.

\item If $b<\frac{1}{2}$ and $t+\frac{1}{2}\le 0$ then all tests are asymptotically powerless.

\item If $b<\frac{1}{2}$ and $t+\frac{1}{2}> 0$ then the conditionally centered max test of Theorem \ref{Centered_max_test} is asymptotically powerful.

\end{enumerate}
\end{thm}

Our second result describes the detection boundary for the degree corrected two star ERGM if $(\theta,\beta_0)\in\Theta_2$.

\begin{thm}\label{Nonunique_dense}
Let $Y$ be an observation from from \eqref{Y:edge_star}, and assume $(\theta,\beta_{0})\in \Theta_2$. Consider the hypothesis testing problem  described in \eqref{Hypothesis_testing} with $s=n^{b}$ and $A=n^t$ for $\in (0,1)$ and $t<0$.

\begin{enumerate}[(a)]
\item If $b\ge \frac{1}{2}$ and $b+t<0$, all tests are asymptotically not powerful.

\item If $b\ge \frac{1}{2}$ and $b+t>0$, then the conditionally centered sum test of Theorem \ref{Centered_mean_test} is asymptotically powerful.

\item If $b<\frac{1}{2}$ and $t+\frac{1}{2}\le 0$ then all tests are asymptotically not powerful.

\item If $b<\frac{1}{2}$ and $t+\frac{1}{2}> 0$ then the conditionally centered max test of Theorem \ref{Centered_max_test} is asymptotically powerful.

\end{enumerate}
\end{thm}


%
%
%

Note that at a qualitative level, the detection boundary in the regimes $\Theta_1$ and $\Theta_2$ are the same. The only difference is that below the detection boundary, in domain $\Theta_1$ Theorem \ref{Unique_dense} shows that all tests are powerless, and in domain $\Theta_2$ Theorem \ref{Nonunique_dense} shows that all tests are asymptotically not powerful. On the other hand, something fundamentally different happens in the critical domain $\Theta_3$, which corresponds to the choice $(\theta,\beta_0)=(1/2,0)$. In this case the optimal testing threshold is significantly lower than the other regimes, and does not depend on whether $b<1/2$ or $b>1/2$. Moreover, this improved performance does not follow from either  Theorem \ref{Centered_mean_test} or \ref{Centered_max_test}. 
In this case a test based on the unconditional sum of degrees attains the optimal detection boundary, for all values of $(s,A)$. This is explained in our final result below.

\begin{thm}\label{Critical Point} 
Let $Y$ be an observation from from \eqref{Y:edge_star}, and assume $(\theta,\beta_{0})=(\frac{1}{2},0)$. Consider the hypothesis testing problem  described in \eqref{Hypothesis_testing}, with $s=n^b$ and $A=n^{t}$ for some $b\in (0,1)$ and $t<0$.

\begin{enumerate}
\item[(a)]
 If  $b+t+\frac{1}{2}<0$, then all tests are asymptotically not powerful.
 
 \item[(b)]
 If $b+t+\frac{1}{2}>0$, then the total degree test $T_n(.)$ defined by
 \begin{align*}
 T_n(G)=&1\text{ if }\sum_{i=1}^nk_i>L_n,\\
 =&0\text{ otherwise}
 \end{align*}
 is asymptotically powerful for some sequence $L_n$ satisfying
$L_{n}\gg n^{3/2}$.
\end{enumerate}
\end{thm}
This demonstrates that the much weaker criterion $b+t+\frac{1}{2}>0$ is enough for detection at criticality, whereas away from criticality we need stronger conditions on $b,t$. Similar phenomenon of improved detection at criticality have been observed for Ising models \cite{mukherjee2018detection, mukherjee2018global, deb2020detecting}. Given that the two star ERGM can be viewed as an Ising model, it is thus not surprising that this continues to hold here. A summary of the detection boundary for the degree corrected two star ERGM is given in figure \ref{fig:two_star} below.

\begin{figure*}[h]
\centering
\hbox{\hspace{3em}\begin{minipage}[c]{1.0\textwidth}
\includegraphics[width=5in,height=3in]{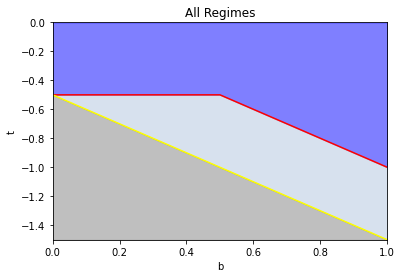}
\end{minipage}}
\caption{\footnotesize{In this figure, we plot $(b,t)$ along $X$ and $Y$ axis respectively, where $s=n^b$ is the size of the signal set, and $A=n^t$ is the magnitude of the signal. The range of $b$ is $(0,1)$, and the range of $t$ is $(-\infty,0)$. The deep blue portion of the plot represents the pairs $(b,t)$ where detection is possible in all regimes $\Theta_1\cup \Theta_2\cup\Theta_3$. The light blue portion of the plot represents the pairs $(b,t)$ where detection is possible  $\Theta_3$, but not for $\Theta_1\cup \Theta_2$. Finally, the grey portion of the plot represents the pairs $(b,t)$ where detection is  impossible in all regimes $\Theta_1\cup \Theta_2\cup\Theta_3$. Also note that in $\Theta_1\cup \Theta_2$ the optimal test depends on whether $b<1/2$ or $b>1/2$, whereas in $\Theta_3$ the optimal test does not depend on $b$. }}
\label{fig:two_star}
\end{figure*}


\subsection{Main Contributions and Future Scope}\label{sec:mc_fs}

In this paper we introduce the degree corrected ERGM, which combines traditional ERGMs with the ${\bm \beta}$-model and thereby allowing for not degree heterogeneity but also dependence between the edges. In this setting, we study the performance of two tests, based on conditionally centered sum of degrees, and conditionally centered maximum degree. The detection rate of these two tests match the performance of the corresponding tests based on the unconditionally centered sum of degree and unconditional maximum degree, respectively, in the independent case $(\theta=0$).  To explore the sharpness of these general tests, we subsequently study the degree corrected two star ERGM in detail. Here we show that in all parameter configurations other than $(\theta,\beta_0)=(1/2,0)$, the optimal detection boundary is attained by one of the conditionally centered tests. At the critical configuration $(\theta,\beta_0)=(1/2,0)$, we show that the optimal detection rate is significantly improved, and this optimal rate is attained by a test based on the unconditionally centered sum of degrees. 
\\

Throughout this paper we assume that the parameters $(\theta,\beta_0)$ are known. If $(\theta,\beta_0)$ is unknown, it may be possible to estimate $(\theta,\beta_0)$ if the signal $(s,A)$ is small, by ignoring the signals altogether and estimating the parameters via the null model MLE/pseudo-likelihood. However such a strategy is hopeless for all values of $(s,A)$, without the knowledge of $(\theta,\beta_0)$. Indeed, consider the following  extreme configuration when $ s=n, A=\infty$, in which case the graph $G$ equals $K_n$ with probability $1$ for any value of $\beta_0$. On the other hand, if $s=A=0$, but $\theta=\infty$, the observed graph is again $K_n$ with probability $1$ for any value of $\beta_0$. Thus having observed $G$, it is impossible to decide whether signal is present or absent, if we are not told the value of $\theta$. It remains to be seen to what extent a partial knowledge of $(\theta,\beta_0)$ can help in our testing problem.
\\

The analysis of the conditionally centered sum and maximum of degrees for general (degree corrected) ERGMs is achieved using concentration results based on the method of exchangeable pairs (\cite{chatterjee2006stein}). Focusing on the degree corrected two star ERGM, to verify the improved detection rate at criticality,  we introduce a continuous auxiliary variable $\phi\in \R^n$ (similar to \cite{mukherjee2018global}), and show that a suitable function of $\phi$ is stochastically much larger under the alternative than under the null hypothesis. Using this, we show that the unconditional sum of degrees is stochastically much larger under the alternative, which gives the improved detection at criticality. The lower bound argument uses the second moment method, which reduces to bounding the correlation between the degrees under the alternative. In the regimes $\Theta_1$ and $\Theta_3$, using GHS inequality (\cite{lebowitz1974ghs}) we can bound the correlations between the edges under the alternative by the correlation under the null, for which bounds are available from \cite{Sumit2020two}, using exchangeability of the null model. In the regime $\Theta_2$ we need to do a conditional second moment argument restricted to the set where the degrees are large. In the absence of a conditional GHS inequality, we have to directly bound the conditional correlations between the edges under the alternative.  To do this, we make crucial use of the auxiliary variable $\phi$ and set up a recursive equation involving the correlations between degrees of the graph. This recursion leads to a uniform bound on the correlations which is also a tight upper bound (in terms of rate), and suffices for the second moment argument. It is of interest to see if one can set up similar recursive equations to bound correlation between edges in general (degree corrected) ERGMs, in presence/absence of auxiliary variables.

In this paper we focus on the optimal detection rates while studying the detection boundary. A natural follow up question is to study existence of sharp constants (depending on $\theta,\beta_0$) which controls the detection boundary for the degree corrected two star ERGMs. Similar to \cite{mukherjee2018detection}, we expect a sharp phase transition (i.e. existence of a constant which determines the optimal detection boundary) in the regime $b<1/2$, when $(\theta,\beta_0)\ne (1/2,0)$. We believe that to attain optimal detection constants, one needs to study a conditionally centered version of the Higher Criticism Test in the regime $1/4<b<1/2$, wheres the maximum test should suffice in the regime $b<1/2$. Going beyond the two star case, it is of interest to find optimal detection rates, both away from, and at, \enquote{criticality}, for general degree corrected ERGMs. A major challenge in carrying out the lower bound argument beyond the two star case is the absence of tight correlation bounds for general ERGMs, both under the null and alternative hypotheses.

\subsection{Outline}

The outline of the paper is as follows. In section \ref{proofs} we verify results Theorems \ref{Centered_mean_test} and \ref{Centered_max_test}. In section \ref{sec:two_star} we verify Theorems \ref{Unique_dense} and \ref{Nonunique_dense}. The proofs of the results of section \ref{sec:two_star} uses some supporting lemmas, the proofs of which is deferred to section \ref{sec:appen}.

\section{Proof of Theorems \ref{Centered_mean_test} and \ref{Centered_max_test}}\label{proofs}
We will need the following concentration bound for conditionally centered linear statistics for proving the results of this section. The proof of this lemma is similar to \cite[Lemma 2.1]{deb2020fluctuations} and \cite[Lemma 1]{mukherjee2018global}. 

\begin{lem}\label{Concen}
Let $G$ be a random graph from the model \eqref{eq:ERGM}. Then for any arbitrary collection of positive numbers $\{c_{e}\}_{e\in \mathcal{E}}$ and any $x>0$ we have
\begin{align}\label{concentration}
\P_{n,\theta,{\bm \beta}}\Big(\Big|\sum\limits_{e\in\mathcal{E}}c_{e}\Big(G_{e}-\E_{n,\theta,{\bm \beta}}\big(G_{e}\big|G_{f}:f\neq e\big)\Big)\Big|>x\Big)\leq 2\exp\left\{-\frac{x^2}{\lambda\sum_{e\in\mathcal{E}}c_{e}^{2}}\right\}
\end{align}
where $\lambda=\lambda(\theta,H)$ is a constant depending only on $\theta>0$ and the subgraph $H$.
\end{lem}
\begin{proof}
Produce an exchangeable pair $(G,G')$ in the following way:

Pick a random vertex pair $I$ of the uniformly from the set $\mathcal{E}$ with cardinality $N={n\choose 2}$. If $I=e$, replace the random variable $G_e$ by $G_e'$ a pick from the conditional distribution given $\{G_f, f\ne e\}$. Let this new graph be denoted by $G'$. It is easy to verify that $(G,G')$ is indeed an exchangeable pair.
Setting $J(G):=\sum_{e\in\mathcal{E}}c_{e}G_{e}$, note that
\begin{align*}
h(G):=\E_{n,\theta,{\bm \beta}}\left(J(G)-J(G')\big|G\right)
=&\frac{1}{N}\sum_{e\in \mathcal{E}} c_e\left(G_e-\E_{n,\theta,{\bm \beta}}\big(G_{e}\big|G_{f}:f\neq e\big)\right)\\
=&\frac{1}{N}J(G)-\frac{1}{N}\sum_{e\in \mathcal{E}} c_e \frac{\exp\Big\{\frac{\theta}{n^{\zeta-2}} N_{e}(H,G)+\beta_e\}}{1+\exp\Big\{\frac{\theta}{n^{\zeta-2}} N_{e}(H,G)+\beta_e\Big\}},
\end{align*}
where $N_e(H,G)$ is the number of copies of $H$ in the graph $G$, which contains the edge $e$. Using the fact that the derivative of the function $\psi(x)=\frac{e^x}{1+e^x}$ is bounded by $\frac{1}{4}$, this gives
\begin{align*}
|h(G)-h(G')|
\le &\frac{|c_{I}|}{N}+\frac{|\theta|}{4N n^{\zeta-2}}\sum_{e\in \mathcal{E}} |c_e| |N_e(H,G)-N_e(H,G')|\\
\le &\frac{|c_{I}|}{N}+\frac{|\theta|}{4N n^{\zeta-2}} \sum_{e\in \mathcal{E}}|c_e| N_{e,I}(H,K_n),
\end{align*}
where $N_{e,f}(H,K_n)$ is the number of copies of $H$ in the complete graph $K_n$ passing through both the edges $e$ and $f$. Consequently, we have
\begin{align*}
\begin{split}
&\Big|\E_{n,\theta,{\bm \beta}}\Big((h(G)-h(G'))(J(G)-J(G'))\Big|G\Big)\Big|\\
&\le \frac{1}{N} \sum_{f\in \mathcal{E}} |c_f|\left[ \frac{|c_f|}{N}+\frac{|\theta|}{4N n^{\zeta-2}} \sum_{e\in \mathcal{E}}|c_e| N_{e,f}(H,K_n)\right]\\
=&\frac{1}{N^2}\sum_{f\in \mathcal{E}} c_f^2+\frac{|\theta|}{4N^2  n^{\zeta-2}} \sum_{e,f\in \mathcal{E}}N_{e,f}(H,K_n)|c_e||c_f|\\
=&\frac{1}{N^2}\sum_{e,f\in \mathcal{E}} B_N(e,f) |c_e| |c_f|,
\end{split}
\end{align*}
where $B_N$ is a $N\times N$ symmetric matrix defined by:
\begin{align*}
B_N(e,f):=\left\{
\begin{array}{rcl}
1 &  \text{if $e=f$}\\
\frac{|\theta|}{4n^{\zeta-2}} N_{e,f}(H,K_n) &  \text{if $e\ne f$}.
\end{array} \right.
\end{align*}
Now for any $e\ne f$ we have
\begin{align*}
N_{e,f}(H,K_n)\lesssim &n^{\zeta-4}\text{ if $e$ and $f$ have no vertex in common},\\
\lesssim &n^{\zeta-3}\text{ if $e$ and $f$ have one vertex in common}.
\end{align*}
This gives
$$\max_{e\in \mathcal{E}}\sum_{f\in \mathcal{E}}B_N(e,f)\lesssim 1+n^2 \frac{1}{n^{\zeta-2}} n^{\zeta-4} +n \frac{1}{n^{\zeta-2}} n^{\zeta-3}\lesssim 1,$$
which in turn implies that the operator norm of the matrix $B_N$ is $O(1)$, and consequently, 
$$\Big|\E_{n,\theta,{\bm \beta}}\Big((h(G)-h(G'))(J(G)-J(G'))\Big|G\Big)\Big|\lesssim \frac{1}{N^2} \sum_{e\in \mathcal{E}} c_e^2\lesssim \frac{1}{n^4} \sum_{e\in \mathcal{E}}c_e^2.$$
Then by Stein's Method for concentration inequalities as in \cite[Theorem 1.5]{chatterjee2006stein}, the conclusion of the lemma follows. 

\end{proof}
\subsection{\textbf{Proof of Theorem \ref{Centered_mean_test}}}

To begin, using Lemma \ref{Concen} with $c_{e}=1$ for all $e\in\mathcal{E}$ gives the existence of a constant $\lambda$ (depending only on $\theta,H$) such that
\begin{align}\label{type1ee}
\P_{\mathcal{H}_{0}}\Big(\Big|\sum\limits_{e\in\mathcal{E}}\big(G_{e}-\E_{n,\bm{\beta}_{0},\theta}\big(G_{e}\big|G_{f}:f\neq e\big)\big)\Big|>L_{n}\Big)\leq2\exp{\left\{-\frac{L_{n}^{2}}{\lambda {n\choose 2}}\right\}}\longrightarrow0,
\end{align}
where the last limit uses $L_n\gg n$. This shows that type I error converges to $0$. 

It thus remains to show that type II error converges to $0$. To this effect, note that $t_e(H,G)\le t_e(H,K_n)$ which is bounded, and so therefore there exist a constant $\delta>0$ such that 
\begin{align}
\label{eq:delta}
\notag\E_{n,\theta,{\bm \beta}}\big(G_{e}\big|G_{f}:f\neq e\big)-\E_{n,\bm{\beta}_{0},\theta}\big(G_{e}\big|G_{f}:f\neq e\big)
=&\psi(\theta t_{e}(H,G)+\beta_i+\beta_j) -\psi(\theta t_e(H,G)+2\beta_0)\\
\geq &\delta\min\{\beta_{i}+\beta_{j}-2\beta_{0},1\}.
\end{align}
Adding this gives
$$\sum_{e\in\mathcal{E}}\Big(\E_{n,\theta,{\bm \beta}}\big(G_{e}\big|G_{f}:f\neq e\big)-\E_{n,\bm{\beta}_{0},\theta}\big(G_{e}\big|G_{f}:f\neq e\big)\Big)\geq \delta nsA.$$
Since $L_n\ll nsA$, for all $n$ large we have
\begin{align*}
\begin{split}
&\P_{n,\theta,{\bm \beta}}\Big(\sum\limits_{e\in\mathcal{E}}\big(G_{e}-\E_{n,\bm{\beta}_{0},\theta}(G_{e}|G_{f}:f\neq e)\big)\leq L_{n}\Big)    \\
\le& \P_{n,\theta,{\bm \beta}}\Big(\Big|\sum\limits_{e\in\mathcal{E}}\big(G_{e}-\E_{n,\theta,{\bm \beta}}(G_{e}|G_{f}:f\neq e)\big)\Big|\geq L_{n}\Big) 
\leq 2\exp\left\{-\frac{L_{n}^{2}}{\lambda n^{2}}\right\}.
\end{split}
\end{align*}
where we again invoke Lemma \ref{Concen} in the last line above. This gives
$$\sup_{\beta\in \Xi(s,A) } \P_{n,\beta,\theta}\Big(\Big|\sum\limits_{e\in\mathcal{E}}\big(G_{e}-\E_{n,\bm{\beta}_{0},\theta}\big(G_{e}\big|G_{f}:f\neq e\big)\big)\Big|\le L_{n}\Big)\le  2\exp\left\{-\frac{L_{n}^{2}}{\lambda {n\choose 2}}\right\},$$
which converges to $0$ as $L_n\gg n$. This completes the proof of the theorem.

\subsection{\textbf{Proof of Theorem \ref{Centered_max_test}}}
As in the previous theorem, it suffices to show that both type I and type II errors converge to $0$. For estimating the type I error, 
using a union bound gives
\begin{align}
\begin{split}
&\P_{\mathcal{H}_{0}}\Big(\max\limits_{1\le i\le n}\Big|\sum_{e\ni i}(G_e-\E_{n,\bm{\beta}_{0},\theta}\big(G_{e}\big|G_{f}:f\neq e\big)\Big{|}>C\sqrt{n\log n}\Big) \\
&\leq\sum\limits_{i=1}^{n}\P_{\mathcal{H}_{0}}\Big(\Big{|}\sum_{e\ni i}(G_e-\E_{n,\bm{\beta}_{0},\theta}\big(G_{e}\big|G_{f}:f\neq e\big)\Big{|}>C\sqrt{n\log n}\Big) 
\le  n\exp{\left\{-\frac{C^{2}n\log n}{\lambda  (n-1)}\right\}},
\end{split}    
\end{align}
where the last inequality uses Lemma \ref{Concen} with $c_e=1$ if $e\ni i$, and $0$ otherwise. For the choice $C>\sqrt{\lambda}$ the RHS above converges to $0$, and so Type I error converges to 0. 
\\

For estimating the Type II error, fix vertex $i$ such that $\beta_i\ge A$. Then using \eqref{eq:delta} gives
$$\sum_{e\ni i}\Big(\E_{n,\theta,{\bm \beta}}\big(G_{e}\big|G_{f}:f\neq e\big)-\E_{n,\bm{\beta}_{0},\theta}\big(G_{e}\big|G_{f}:f\neq e\big)\Big)\geq \delta n\min\{A,1\}.$$
Since $A\ge \kappa \sqrt{\frac{\log n}{n}}$, for all $n$ large we have
$$\delta n\min\{A,1\}\ge \delta \kappa  \sqrt{n\log n} \ge  2C\sqrt{\log n}$$
for the choice $\kappa=\frac{2C}{\delta}$. This gives
\begin{align*}
\begin{split}
&\P_{n,\theta,{\bm \beta}}\Big(\sum_{e\ni i}(G_e-\E_{n,\theta,\beta_0{\bf 1}}\big(G_{e}\big|G_{f}:f\neq e\big)\le C\sqrt{n\log n}\Big)\\
\le &\P_{n,\theta, \bm{\beta}}\Big(\Big{|}\sum_{e\ni i}(G_e-\E_{n,\theta,{\bm \beta}}\big(G_{e}\big|G_{f}:f\neq e\big)\Big{|}\ge C\sqrt{n\log n}\Big)\le 2\exp{\left\{-\frac{C^{2}n\log n}{\lambda  (n-1)}\right\}},
\end{split}    
\end{align*}
where the last inequality again uses Lemma \ref{Concen}.  Thus we have shown $$\sup_{\beta\in \Xi(s,A) } \P_{n,\beta,\theta}\Big(\Big|\sum\limits_{e\ni i}\big(G_{e}-\E_{n,\bm{\beta}_{0},\theta}\big(G_{e}\big|G_{f}:f\neq e\big)\big)\Big|\le C\sqrt{n\log n}\Big)\le  2\exp{\left\{-\frac{C^{2}n\log n}{\lambda  (n-1)}\right\}},$$
which converges to $0$ as before for the choice $C>\sqrt{\lambda}$.

%
%

\subsection{Proof of parts (b) and (d) of Theorem \ref{Unique_dense} and Theorem \ref{Nonunique_dense}}

Part (b) follows by a direct application of Theorem \ref{Centered_mean_test}, on noting that $sA=n^{b+t}\to \infty$ if $b+t>0$. Similarly, part (d) follows by a direct application of Theorem \ref{Centered_max_test}, on noting that $A=n^{t}\gg \sqrt{\frac{\log n}{n}}$ if $t>-\frac{1}{2}$. Both Theorem \ref{Centered_mean_test} and Theorem \ref{Centered_max_test} were proved for $\{0,1\}$ valued random variables, but essentially the same proof goes through for $\{-1,1\}$ valued random variables.

\subsection{Proof of Theorem \ref{Critical Point} part (b)}

To prove Theorem \ref{Critical Point} part (b) (as well as parts (a) and (c) of Theorem \ref{Nonunique_dense} later), we express the two star model as a mixture of $\beta$ models by introducing auxiliary variables, as done in \cite{park2004solution,Sumit2020two}. Suppose $Y$ be a random graph from degree corrected two-star model \eqref{Y:edge_star}. Conditional on $Y$, let$(\phi_1,\cdots,\phi_n)$ be mutually independent components, with
\begin{align}\label{auxil}
\phi_i\sim N\Big(\frac{k_{i}}{n-1},\frac{1}{\theta(n-1)}\Big).    
\end{align}
The joint distribution of $(\phi,Y)$ is computed in the following Proposition. The proof of this is deferred to the appendix (section \ref{sec:appen}). 
\begin{ppn}\label{thm:bayes}
\begin{enumerate}
\item[(a)]
Given $\phi$, the random variables $(Y)_{1\le i<j\le n}$ are mutually independent, with 
\begin{align*}
\P_{n,\theta,{\bm \beta}}(Y_{ij}=1|\phi)=\frac{e^{\theta(\phi_{i}+\phi_{j})+\frac{1}{2}(\beta_{i}+\beta_{j})}}{e^{\theta(\phi_{i}+\phi_{j})+\frac{1}{2}(\beta_{i}+\beta_{j})}+e^{-\theta(\phi_{i}+\phi_{j})-\frac{1}{2}(\beta_{i}+\beta_{j})}}.
\end{align*}
\item[(b)]
The marginal density of $\phi$ (w.r.t. Lebesgue measure) is proportional to
\begin{align}\label{definef}
f_{n,\theta,{\bm \beta}}(\phi):=&\exp\left\{-\sum\limits_{i<j} p_{ij}(\phi_i,\phi_j)\right\},
\end{align}
where $ p_{ij}(x,y)$ equals
\begin{align}
\begin{split}\label{definep}
&\frac{\theta}{2}(x^2+y^2)-\log\cosh{[\theta(x+y)+\frac{1}{2}(\beta_{i}+\beta_{j})]}\\
 =&\frac{\theta}{4}(x-y)^2+q\Big(\frac{x+y}{2}\Big)+\log\cosh\Big(\theta(x+y)\Big)-\log\cosh\Big(\theta(x+y)+\frac{1}{2}(\beta_{i}+\beta_{j})\Big),
 \end{split}
 \end{align}
 with
\begin{align}\label{eq:q}
q(x):=\theta x^2-\log\cosh(2\theta x).
\end{align}
\end{enumerate}
\end{ppn}

We now state the following lemma, which is the analogue of \cite[Lemma 4.1]{Sumit2020two}. The proof of these lemmas are deferred to the appendix (\ref{sec:appen}).

%
%
%
%


\begin{lem}\label{1.3}
Suppose $\theta=1/2$, and ${\bm \beta}\in [0, n^{-1/2}]^n$. Then for  any positive integer $\ell\in \N$, there exist a constant $C$ depending only on $\ell, \theta$ such that $$\max_{1\le i\le n}\E_{n,\theta,{\bm \beta}}|\phi_{i}-\Bar{\phi}|^{l}\leq Cn^{-l/2}.$$
\end{lem}

\begin{proof}[Proof of Theorem \ref{Critical Point} part (b)]

We begin by claiming the existence of a sequence of positive reals $K_n\to\infty$ such that
\begin{align}\label{lem:mean_alt}
\lim_{n\to\infty}\sup_{{\bm \beta}\in \Xi(s,A)}\P_{n,\theta,{\bm \beta}}(\tanh(\bar{\phi})\le n^{-1/4}K_n)=0.
\end{align}
Given \eqref{lem:mean_alt}, we first finish the proof of the theorem. Note that
\begin{align}
\begin{split}
\sum\limits_{i<j}\Big[Y_{ij}-\tanh(\bar{\phi})\Big]=&\sum\limits_{i<j}\Big[\tanh\Big(\frac{\phi_{i}+\phi_{j}}{2}+\frac{\beta_{i}+\beta_{j}}{2}\Big)-\tanh(\bar{\phi})\Big]\\
\ge &\sum\limits_{i<j}\Big[\tanh\Big(\frac{\phi_{i}+\phi_{j}}{2}\Big)-\tanh(\bar{\phi})\Big]\\
\gtrsim &- \sum_{i<j}\Big(\frac{\phi_{i}+\phi_{j}}{2}-\bar{\phi}\Big)^2\gtrsim -n\sum_{i=1}^n(\phi_i-\bar{\phi})^2.
\end{split}
\end{align}
Using \eqref{lem:mean_alt} and Lemma \ref{1.3} along with the above display we have
$$\lim_{n\to\infty}\sup_{{\bm \beta}\in \Xi(s,A)}\P_{n,\theta,{\bm \beta}}\Big(\sum\limits_{i<j}Y_{ij}\le n^{3/2} K_n)=0.$$
and so Type II error converges to $0$.
Since
\begin{align*}
\begin{split}
\P_{\mathcal{H}_{0}}\left(\sum\limits_{i<j}Y_{ij}>n^{\frac{3}{2}}K_{n}^{\frac{1}{2}}\right) \to 0.
\end{split}
\end{align*}
using \cite[Theorem 1.1]{Sumit2020two}, Type I error converges to $0$ as well. This shows that the test which rejects for large values of $\sum_{i<j}Y_{ij}$ is asymptotically powerful.
\\

It thus remains to verify \eqref{lem:mean_alt}. To this end, 
 assume without loss of generality that 
\begin{align*}
\beta_i=&A\text{ if }1\le i\le s\\
=&0\text{ if }s+1\le i\le n,
\end{align*}
where $A=n^{t}$. Also if $b+t+1/2>0$, replacing $t$ by $t':=\min(t,-1/2)$ we have $$b+t'+1/2=\min\Big(b+t+1/2, b-\frac{1}{2}+\frac{1}{2}\Big)=\min(b+t+1/2,b)>0.$$
 Since the distribution of $\bar{\phi}$ is stochastically increasing in $A$,  without loss of generality by replacing $t$ by $t'$ if necessary we can assume $t\le -\frac{1}{2}$, which gives $A\le n^{-1/2}$.   
 Using Taylor's series expansion twice, we have
\begin{align*}
&\log \cosh\Big(\frac{\phi_i+\phi_j}{2}+\frac{\beta_i+\beta_j}{2}\Big)-\log\cosh\Big(\frac{\phi_i+\phi_j}{2}\Big)\\
=&\frac{\beta_i+\beta_j}{2}\tanh\Big(\frac{\phi_i+\phi_j}{2}\Big)+O (\beta_i+\beta_j)^2\\
=&\frac{\beta_i+\beta_j}{2}\tanh( \bar{\phi})+O\Big((\beta_i+\beta_j)|\phi_i+\phi_j-2\bar{\phi}|\Big)+O(\beta_i+\beta_j)^2.
\end{align*}
Summing over $i<j$ and using \eqref{definef} and \eqref{definep} we get
\begin{align}\label{eq:p_recall}
\begin{split}
-\log f_{n,\theta,{\bm\beta}}(\phi)=&-\log f_{n,\theta,{\bm 0}}(\phi)-\frac{(n-1)sA}{2}\tanh( \bar{\phi})\\
+&O\left(nA\sum_{i=1}^s|\phi_i-\bar{\phi}|+sA \sum_{i=1}^n|\phi_i-\bar{\phi}|+nsA^2\right),
\end{split}
\end{align}
where 
\begin{align}\label{eq:fh_0}
\notag-\log f_{n,\theta,{\bm 0}}(\phi):=&\sum_{i<j}\Big[ \frac{1}{8}(\phi_i-\phi_j)^2+q\Big(\frac{\phi_i+\phi_j}{2}\Big)\Big]\\
=&\frac{n}{8}\sum_{i=1}^n(\phi_i-\bar{\phi})^2+\sum_{i<j} q\Big(\frac{\phi_i+\phi_j}{2}\Big),
\end{align}
with $q(.)$ as in \eqref{eq:q}.
As the notation above suggests, $f_{n,\theta,{\bm 0}}$ defined above is the (unnormalized) density of $\phi$ under $\mathcal{H}_0$.
Using \eqref{eq:p_recall}, along with Lemma \ref{1.3} we have
\begin{align*}
-\log f_{n,\theta,{\bm\beta}}(\phi)=-\log f_{n,\theta,{\bm 0}}(\phi)-\frac{nsA}{2}\tanh( \bar{\phi})-R_n, 
\end{align*}
where $$\E_{n,\theta,{\bm \beta}}|R_n|\lesssim \sqrt{n} s A+nsA^2\lesssim \sqrt{n}sA$$ using $A\le n^{-1/2}$.
Thus, for any $K$ fixed and $K_n':=n^{3/4}sA$ we have
\begin{align*}
&\P_{n,\theta,{\bm \beta}}(\tanh( \bar{\phi})<Kn^{-1/4}) \\
\leq &\P_{n,\theta,{\bm \beta}}(|R_n|>K_n')+\P_{n,\theta,{\bm \beta}}(\tanh( \bar{\phi})<K n^{-1/4}, |R_n|\le K_n')\\
\le & \P_{n,\theta,{\bm \beta}}(|R_n|>K_n')+ e^{K_n'} \frac{\E_{\mathcal{H}_0} \exp\Big[\frac{nsA}{2}\tanh( \bar{\phi})\Big]1\Big\{\tanh(\bar{\phi})<K n^{-1/4}\Big\}}{\E_{\mathcal{H}_0} \exp\Big[\frac{nsA}{2}\tanh( \bar{\phi})\Big]1\Big\{|R_n|\le K_n'\Big\}}\\
\le &\P_{n,\theta,{\bm \beta}}(|R_n|>K_n')+ \frac{e^{K_n'+ \frac{Kn^{3/4}sA}{2}-\frac{nsA \tanh(2 Kn^{-1/4})}{2}}}{ \P_{\mathcal{H}_0} \Big(\bar{\phi}>2Kn^{-1/4}, |R_n|\le K_n'\Big)}
\end{align*}
On letting $n\to\infty$ and noting that $K_n'=n^{3/4} sA\gg \sqrt{n} sA2$ we have
$$\lim_{n\to\infty} \sup_{{\bm \beta}\in \Xi(s,A)}\P_{n,\theta,{\bm \beta}}(|R_n|\le K_n')=0, \text{ and } \lim_{n\to \infty} \P_{\mathcal{H}_0} \Big(\bar{\phi}>2Kn^{-1/4}, |R_n|\le K_n'\Big)=\P(\zeta>2K)>0,$$
where $\zeta $ has density proportional to $e^{-\zeta^4/12-\zeta^2/24}$ (c.f.~\cite[Lemma 4.2]{Sumit2020two}).
Combining the last two displays we have
$$\lim_{n\to\infty} \sup_{{\bm \beta}\in \Xi(s,A)}\P_{n,\theta,{\bm \beta}}(\tanh( \bar{\phi})<Kn^{-1/4}) =0.$$
Since this holds for every fixed $K$, there exists $K_n\to\infty$ such that
$$\limsup_{n\to\infty}\sup_{{\bm \beta}\in \Xi(s,A)} \P_{n,\theta,{\bm \beta}}(\tanh( \bar{\phi})<K_nn^{-1/4}) =0.$$
This verifies \eqref{lem:mean_alt}, and hence completes the proof of the theorem.

\end{proof}

\section{Proof of  parts (a) and (c) Theorems \ref{Unique_dense} and \ref{Nonunique_dense}}\label{sec:two_star} 

With $\Xi(s,A)$ as defined in \eqref{Alternative}, consider the following subset of $\Xi(s,A)$.
\begin{align}\label{alter}
\Tilde{\Xi}(s,A):=\Big\{\bm{\beta}=\beta_{0}{\bf 1}+\bm{\mu}:|\text{supp}(\bm{\mu})|=s, and\ \mu_{i}=A, i\in\text{supp}(\bm{\mu})\Big\}.
\end{align}
Let $\pi(d\bm{\beta})$ be a prior on ${\Xi}(s,A)$, which put probability mass $1/\tbinom{n}{s}$ on each of configurations in $\Tilde{\Xi}(s,A)$. And let  $\Q_{\pi}(.):=\int\P_{n,\theta,{\bm \beta}}(.)\pi(d\bm{\beta})$ denote the marginal distribution of $Y$ under this prior. To show that all tests for the problem \eqref{Hypothesis_testing} are asymptotically powerless, using the second moment method it suffices to show that
\begin{align}\label{Likelihood_Ratio}
\lim_{n\to\infty}\E_{\mathcal{H}_0}L_\pi(Y)^2=1, \text{ where }L_{\pi}(Y):=\frac{\Q_{\pi}(Y)}{\P_{\mathcal{H}_{0}}(Y)}
\end{align}
is the likelihood ratio. The following lemma gives an upper bound to the second moment of $L_\pi(.)$.
\begin{lem}\label{LBlemma}
For any $(\theta,\beta_{0})$, with $L_\pi(.)$ as defined in \eqref{Likelihood_Ratio} we have 
\begin{align}\label{boundboundbound}
\E_{\mathcal{H}_{0}}L_{\pi}^{2}(Y)\leq\exp\left\{A^{2}s^{2}Cov_{\bm{\beta} = (\beta_{0}/2)\bm{1}}(k_{1},k_{2})+\frac{2s^{2}}{n}(e^{A^{2}Var_{\bm{\beta} = (\beta_{0}/2)\bm{1}}(k_{1})}-1)\right\},
\end{align}
whenever $n>2s$.
\end{lem}
\begin{proof}
Define $\Lambda_s:=\{S\big{|}S\subset\{1,2,...,n\}, |S|=s\}$. For any $S\in\Lambda_s$, define a vector $\bm{\beta}_{S}$ by setting
\begin{align*}
 \beta_{S,i}=&\beta_{0}+A\text{ if }i\in S,\\
 =&\beta_{0}\text{ if }i\notin S.
 \end{align*} 
By symmetry, the normalizing constant $Z_{n}(\bm{\beta}_{S},\theta)$ is the same for all $S\in\Lambda_s$, which we denote by $Z_{n}(\bm{\beta}_{[s]},\theta)$ for the rest of this proof.
Then, a direct calculation gives
\begin{align}\label{calculation}
\notag\E_{\mathcal{H}_{0}}L_{\pi}^{2}(Y)&=\frac{Z_{n}^{2}(\beta_{0},\theta)}{Z_{n}^{2}(\bm{\beta}_{[s]},\theta)}\frac{1}{\binom{n}{s}^{2}}\E_{\mathcal{H}_{0}}\sum\limits_{S_{1},S_{2}\in\Lambda_s}e^{\sum\limits_{j\in S_{1}}\frac{A}{2}k_{j}+\sum\limits_{j\in S_{2}}\frac{A}{2}k_{j}} \\
\notag&=\frac{Z_{n}(\beta_{0},\theta)}{Z_{n}^{2}(\bm{\beta}_{[s]},\theta)}\frac{1}{\binom{n}{s}^{2}}\sum\limits_{S_{1},S_{2}\in\Lambda}\frac{Z_{n}(\bm{\beta}_{S_{1}}+\bm{\beta}_{S_{2}},\theta)}{Z_{n}(\bm{\beta}_{S_{1}}+\bm{\beta}_{S_{2}},\theta)}\sum\limits_{Y}e^{\frac{\theta}{2n}\sum\limits_{i=1}^{n}k_{i}^{2}+\sum\limits_{j=1}^{n}\frac{\beta_{S_{1},j}+\beta_{S_{2},j}}{2}k_{j}}\\
&=\frac{1}{\binom{n}{s}^{2}}\sum\limits_{S_{1},S_{2}\in\Lambda}\frac{Z_{n}(\beta_{0},\theta)Z_{n}(\bm{\beta}_{S_{1}}+\bm{\beta}_{S_{2}},\theta)}{Z_{n}(\bm{\beta}_{S_{1}},\theta)Z_{n}(\bm{\beta}_{S_{2}},\theta)}=\frac{1}{\binom{n}{s}^{2}}\sum\limits_{S_{1},S_{2}\in\Lambda}R_{S_1,S_2},
\end{align}
where 
\begin{align*}
R_{S_{1},S_{2}}:=&\log\left(\frac{Z_{n}(\beta_{0},\theta)Z_{n}(\bm{\beta}_{S_{1}}+\bm{\beta}_{S_{2}},\theta)}{Z_{n}(\bm{\beta}_{S_{1}},\theta)Z_{n}(\bm{\beta}_{S_{2}},\theta)}\right)\\
=&\log Z_{n}(\bm{\beta}_{S_{1}}+\bm{\beta}_{S_{2}},\theta)-\log Z_{n}(\bm{\beta}_{S_{2}},\theta)-\log Z_{n}(\bm{\beta}_{S_{1}},\theta)+\log Z_{n}(\beta_{0},\theta).
\end{align*}
Setting $W=S_{1}\bigcap S_{2}$, note that $R_{S_1,S_2}$ only depends on $|W|$ by symmetry. Thus, without loss of generality we assume that $S_{1}=\{1,2,3,...,s\}$ and $S_{2}=\{1,2,...,w,s+1,s+2,...,2s-w\}$. Consequently we have
\begin{align*}
R_{S_{1},S_{2}}
=&\sum_{j\in S_1}\Big[\log Z_{n}(\bm{\beta}_{[j]}+\bm{\beta}_{S_{2}},\theta)-\log Z_{n}(\bm{\beta}_{[j-1]}+\bm{\beta}_{S_{2}},\theta)-\log Z_{n}(\bm{\beta}_{[j]},\theta)+\log Z_{n}(\bm{\beta}_{[j-1]},\theta)\Big],
\end{align*}
where $\bm{\beta}_{[j]}$ denotes the vector $\bm{\beta}$ which equals $A$ on first $j$ entries, and $\beta_{0}$ for rest of its entries, 
The summand in the RHS above equals 
\begin{align*}
&\log Z_{n}(\bm{\beta}_{[j]}+\bm{\beta}_{S_{2}},\theta)-\log Z_{n}(\bm{\beta}_{[j-1]}+\bm{\beta}_{S_{2}},\theta)-\log Z_{n}(\bm{\beta}_{[j]},\theta)+\log Z_{n}(\bm{\beta}_{[j-1]},\theta) \\
&=\int_{0}^{A}\frac{\partial\log Z_{n}(\bm{\beta}_{[j-1]}+\bm{\beta}_{S_{2}}+\gamma\mathbf{e_{j}},\theta)}{\partial\beta_{j}}d\gamma-\int_{0}^{A}\frac{\partial\log Z_{n}(\bm{\beta}_{[j-1]}+\gamma\mathbf{e_{j}},\theta)}{\partial\beta_{j}}d\gamma \\
&=\int_{0}^{A}A\sum\limits_{r\in S_{2}}\frac{\partial\log Z_{n}(\bm{\beta}_{[j-1]}+\bm{\xi}+\gamma\mathbf{e_{j}})}{\partial\beta_{j}\partial\beta_{r}}|_{\bm{\xi}\preceq\bm{\beta}_{S_{2}}}d\gamma\\
&=\int_{0}^{A}A\sum\limits_{r\in S_{2}}Cov_{\bm{\beta}=\bm{\beta}_{[j-1]}+\bm{\xi}+\gamma\mathbf{e_{j}}}(k_{j},k_{r})d\gamma 
\end{align*}
If $A\to 0$, then ${\bm \beta}\ge {\bf 0}$ if $\beta_0\ge 0$, and ${\bm \beta}\le {\bf 0}$ for all $n$ large if $\beta_0<0$. Note that the GHS inequality \cite{lebowitz1974ghs} holds if either ${\bm \beta}\ge {\bf 0}$ or ${\bm \beta}\le {\bf 0}$ (the second conclusion follows on noting that $Cov_{\bm \beta}(k_r,k_s)=Cov_{\bm \beta}(-k_r,-k_s)$,
thereby giving $$Cov_{\bm{\beta}=\bm{\beta}_{[j-1]}+\bm{\xi}+\gamma\mathbf{e_{j}}}(k_{j},k_{r})\le Cov_{\bm{\beta}=(\beta_{0}/2)\bm{1}}(k_{j},k_{r}).$$
  Combining the above two displays, this gives
\begin{align*}
R_{S_{1},S_{2}}&\leq\sum\limits_{j\in S_{1}}\int_{0}^{A}A\sum\limits_{r\in S_{2}}Cov_{\bm{\beta}=(\beta_{0}/2)\bm{1}}(k_{j},k_{r})d\gamma \\
&=A^{2}w Var_{\bm{\beta}=(\beta_{0}/2)\bm{1}}(k_{1})+A^{2}(s^{2}-w)Cov_{\bm{\beta}=(\beta_{0}/2)\bm{1}}(k_{1},k_{2}).
\end{align*}
Along with \eqref{calculation}, this further gives
$$\E_{\mathcal{H}_{0}}L_{\pi}^{2}(Y)\leq\exp{\{A^{2}s^{2}Cov_{\bm{\beta}=(\beta_{0}/2)\bm{1}}(k_{1},k_{2})\}}\E_{W}\exp\{A^{2} Var_{\bm{\beta}=(\beta_{0}/2)\bm{1}}(k_{1}) W\} $$
where $W$ follows Hypergeometric distribution with parameters $(n,s,s)$. Since $2s<n$, $W$ is stochastically dominated by a binomial distribution with parameters $\Big(s,\frac{s}{n-s}\Big)$ (\cite[Lemma 6.1]{mukherjee2018detection}), which gives
$$\E_{W}\exp\{A^{2} Var_{\bm{\beta}=(\beta_{0}/2)\bm{1}}(k_{1}) W\}\leq\exp\left\{\frac{2s^{2}}{n}(e^{A^{2}Var_{\bm{\beta}=(\beta_{0}/2)\bm{1}}(k_{1})}-1)\right\}.$$
Combining the last two displays, we have verified \eqref{boundboundbound}.
\end{proof}
\subsection{\textbf{Proof of Parts (a) and (c) of Theorem \ref{Unique_dense}}}
With $L_{\pi}$ as in defined in \eqref{Likelihood_Ratio},
it suffices to show that $$\lim_{n\rightarrow\infty}\E_{\mathcal{H}_{0}}L_{\pi}^{2}(Y)=1.$$
 By  \cite[Lemma 4.4]{Sumit2020two} we have
\begin{align}
Var_{\bm{\beta}=(\beta_{0}/2)\bm{1}}(\sum_{e\in\mathcal{E}}Y_{e})\lesssim n^{2},    
\end{align}
which gives the existence of a constant $c$ depending on $\theta$ such that
\begin{align}
Var_{\bm{\beta}=(\beta_{0}/2)\bm{1}}(k_{1})\le cn,\quad Cov_{\bm{\beta}=(\beta_{0}/2)\bm{1}}(k_{1},k_{2})\le c   
\end{align}
Using this along with Lemma \ref{LBlemma} gives
\begin{align}\label{eq:LBlemma}
\E_{\mathcal{H}_{0}}L_{\pi}^{2}(Y)\le \exp\left\{cA^2 s^2 +\frac{2s^2}{n}(e^{cA^2n}-1)\right\}.
\end{align}

\subsubsection{Proof of part (a)}
 In this regime we have $s=n^b$ and $A=n^t$ with $b\ge \frac{1}{2}$ and $b+t<0$. This gives $$\max(A^2 n,A^2s^2)=\max(n^{2t+1}, n^{2t+2b})=n^{2t+2b}\to 0,$$ using which the exponent in the RHS of \eqref{eq:LBlemma}  converges to $0$. This completes the proof of part (a).
 
 \subsubsection{Proof of part (c)}
  In this regime we have $s=n^b$ and $A=n^t$ with $b< \frac{1}{2}$ and $t\le -\frac{1}{2}$. This gives $A^2s^2=n^{2b+2t}\to 0$. Also $$\frac{s^2}{n} e^{cA^2n-1}\le e^{c-1}\frac{s^2}{n} = e^{c-1}n^{2b-1}\to 0.$$
  Consequently, the RHS of \eqref{eq:LBlemma}  again converges to $0$. This completes the proof of part (c).

   \subsection{\textbf{Proof of Theorem \ref{Critical Point} Part (a)}}

As before, with $L_{\pi}$ defined in \eqref{Likelihood_Ratio}, it is sufficient to show that $$\lim_{n\rightarrow\infty}\E_{\mathcal{H}_{0}}L_{\pi}^{2}(Y)=1.$$
To this effect, using Theorem 2.4 \& Lemma 4.8 in \cite{Sumit2020two} we get
\begin{align*}
Var_{\mathcal{H}_{0}}(\sum_{e\in\mathcal{E}}Y_{e})\lesssim n^{3}.
\end{align*}
Along with the non-negativity of covariance, this gives 
\begin{align*}
Cov_{\mathcal{H}_{0}}(k_{1},k_{2})=O(n).
\end{align*}
For getting the optimal bound on $Var_{\mathcal{H}_{0}}(k_{1})$, use \eqref{auxil} to get
\begin{align*}
Var_{\mathcal{H}_{0}}(k_1)\lesssim n^2Var_{\mathcal{H}_{0}}(\phi_1)+n
\lesssim& n^2\Big[Var_{\mathcal{H}_{0}}(\bar{\phi})+Var_{\mathcal{H}_{0}}(\phi_1-\bar{\phi})\Big]+n\lesssim n,
\end{align*}
where the last inequality uses \cite[Lemma 4.1]{Sumit2020two}.
Combing the above two displays along with Lemma \ref{LBlemma} gives the existence of a constant $c$ free of $n$, such that
\begin{align}\label{eq:LBlemma2}
\E_{\mathcal{H}_{0}}L_{\pi}^{2}(Y)\le \exp\left\{cA^2 s^2n +\frac{2s^2}{n}(e^{cA^2n}-1)\right\}.
\end{align}

 
Now, recall that in this regime we have $s=n^b$ and $A=n^t$ with $b+t+\frac{1}{2}<0$. This gives $A^2s^2 n=n^{2b+2t+1}\to 0$. Also, noting that $2t+1<0$ we have
 $$\frac{s^2}{n}(e^{cA^2n}-1)\le  n^{2b-1} (e^{c n^{2t+1}}-1)\lesssim n^{2b+2t+1}\to 0.$$
 Along with \eqref{eq:LBlemma2}, this gives $\lim_{n\to\infty}\E_{\mathcal{H}_{0}}L_{\pi}^{2}(Y)=1$. This completes the proof of part (b).

\subsection{\textbf{Proof of Theorem \ref{Nonunique_dense} parts (a) and (c)}}
We first state the following lemma about the function $q(.)$ introduced in \eqref{eq:q}, the proof of which follows from straightforward calculus (see for e.g. \cite{dembo2010gibbs}). 

\begin{lem}\label{lem:calculus}
If $\theta>1/2$, the equation $q'(x)=2\theta[ x-\theta \tanh(2\theta x)]$  has a unique positive root $t$, say, on $(0,\infty)$. Further, $t$ is the unique global minimizer of $q(.)$ on $[0,\infty)$. 
\end{lem}

We will use the notation $t$ introduced in the above lemma throughout the rest of the paper. Set
\begin{align}\label{eq:u}
 U:=\cap_{i=1}^n V_i, \quad V_{i}:=\{Y:k_i\ge (n-1)t/2\}.
 \end{align}
Restricting the probability measure \eqref{Y:edge_star} to the set $U$, define the probability measure $\P_{n,\bm{\beta},U}(\cdot)$ by setting
\begin{align}\label{restricted_measure}
\P_{n,\bm{\beta},U}(Y)=\frac{1}{Z^{+}_{n}(\bm{\beta},\theta)}\exp\left\{\frac{\theta}{2n}\sum\limits_{i=1}^{n}k_{i}^{2}+\frac{1}{2}\sum\limits_{i=1}^{n}\beta_{i}k_{i}\right\}1\{Y\in U\}.
\end{align}
where $$Z^{+}_{n}(\bm{\beta},\theta)=\sum_{Y\in U}\exp\left\{\frac{\theta}{2n}\sum\limits_{i=1}^{n}k_{i}^{2}+\frac{1}{2}\sum\limits_{i=1}^{n}\beta_{i}k_{i}\right\}$$
is the restricted normalizing constant.
As before, consider the sub parameter space $\Tilde{\Xi}(s,A)$ defined in \eqref{alter}, let $\pi(d\bm{\beta})$ be a prior on $\Tilde{\Xi}(s,A)$, which put probability mass $1/\tbinom{n}{s}$ on each of configurations in $\Tilde{\Xi}(s,A)$. And let  $\Q_{\pi,U}(.):=\int\P_{n,\bm{\beta},U}(.)\pi(d\bm{\beta})$ denote the mixed alternative distribution of $Y$. 
Since \cite[Lem 4.3]{Sumit2020two} gives $\P_{\mathcal{H}_0}(U)\to 1/2$, to verify the absence of asymptotically powerful tests setting
\begin{align}\label{ratio}
L_{\pi,U}(Y):=\frac{\Q_{\pi,U}(Y)}{\P_{\mathcal{H}_{0},U}(Y)},
\end{align}
it suffices to show:
\begin{align}\label{sufficient_condition}
\E_{\mathcal{H}_{0},U}L_{\pi,U}^{2}(Y)\to 1.
\end{align}
Proceeding similar to Lemma \ref{LBlemma}, we get 
\begin{align}
\begin{split}\label{eq:2nd}
\E_{\mathcal{H}_{0},U}L_{\pi,U}^{2}(Y)&
=\frac{1}{\binom{n}{s}^{2}}\sum\limits_{S_{1},S_{2}\in\Lambda}\frac{Z_{n}^{+}(0,\theta)Z_{n}^{+}(\bm{\beta}_{S_{1}}+\bm{\beta}_{S_{2}},\theta)}{Z_{n}^{+}(\bm{\beta}_{S_{1}},\theta)Z_{n}^{+}(\bm{\beta}_{S_{2}},\theta)}.
\end{split}   
\end{align}
Setting $R^+_{S_{1},S_{2}}$ as
$$R^+_{S_{1},S_{2}}:=\left(\log Z_{n}^{+}(\bm{\beta}_{S_{1}}+\bm{\beta}_{S_{2}},\theta\log Z_{n}^{+}(\bm{\beta}_{S_{2}},\theta))-(\log Z_{n}^{+}(\bm{\beta}_{S_{1}},\theta)-\log Z_{n}^{+}(0,\theta)\right).$$
A Taylor's series expansion gives 
\begin{align}\label{RSS}
R_{S_{1},S_{2}}
=A^{2}\sum\limits_{i\in S_{1}}\sum\limits_{j\in S_{2}}Cov_{\bm{\delta}=\alpha{\bf 1}_{S_{1}}+\gamma {\bf 1}_{S_{2}}}(k_{i},k_{j}|U)
\end{align}
where $\alpha,\gamma\in(0,A)$ and ${\bf 1}_{S}$ denote vector having unit signals at $S$, and $\bm{\delta}:=\alpha{\bf 1}_{S_{1}}+\gamma{\bf 1}_{S_{2}}\in [0, 2 n^{-1/2}]^n$. We now claim that
\begin{lem}\label{lem:cov}
$$\max_{1\le i<j\le n}\sup_{{\bm \beta}\in [0, 2n^{-1/2}]^n}Cov_{{\bm \beta}}(k_i,k_j|U)\lesssim 1.$$
\end{lem}
We defer the proof of Lemma \ref{lem:cov} to the end of the section.
Finally, use Lemma \ref{2.3} to conclude that 
\begin{align}\label{varboundb}
\max_{1\le i\le n}Var_{\bm \delta}(k_i|U)\lesssim n.
\end{align} 
Given Lemma \ref{lem:cov} along with \eqref{varboundb} and \eqref{RSS},  we have the existence of a constant $C$ free of $n$ such that 
$$R_{S_{1},S_{2}}\leq CWA^{2}n+Cs^{2}A^{2},$$
which along with \eqref{eq:2nd} gives
\begin{align}
\E_{\mathcal{H}_{0},U}L_{\pi,U}^{2}(Y)\leq   \exp\{CA^{2}s^{2}\}\E_{W}\exp{\{CA^{2}nW\}}
\end{align}
where $W$ follows Hypergeometric distribution with parameters $(n,s,s)$. As before, using the fact that $n>2s$, $W$ is stochastically dominated by a binomial distribution with parameters $(s,\frac{s}{n-s})$. This gives
\begin{align}\label{finalbound}
\E_{\mathcal{H}_{0},U}L_{\pi,U}^{2}(Y)\leq\exp{\{CA^{2}s^{2}+\frac{2s^{2}}{n}(e^{CA^{2}n}-1)\}}.
\end{align}

\subsubsection{Proof of Theorem \ref{Nonunique_dense} part (a)}

In this regime we have $s=n^b$ and $A=n^t$ with $b\ge \frac{1}{2}$ and $b+t<0$. This gives  $A^2 s^2 =n^{2t+2b}\to 0.$
Also we have $A^2n=n^{2t+1}\to 0$, and so
$$\frac{s^2}{n} (e^{CA^{2}n}-1)\lesssim s^2A^2=n^{2b+2t}\to 0.$$

Combining the above two displays with \eqref{finalbound}, we have $\E_{\mathcal{H}_{0},U}L_{\pi,U}^{2}(Y)\to 1$, as desired. This completes the proof of part (a).

\subsubsection{Proof of Theorem \ref{Nonunique_dense} part (c)}

In this regime we have $s=n^b$ and $A=n^t$ with $b< \frac{1}{2}$ and $t+\frac{1}{2}<0$. This gives $$A^2 s^2 =n^{2t+2b}\le n^{2t+1} \to 0.$$
Also we have $A^2n=n^{2t+1}\to 0$, and so
$$\frac{s^2}{n} (e^{CA^{2}n}-1)\lesssim s^2A^2=n^{2b+2t}\to 0.$$

Combining the above two displays with \eqref{finalbound}, we have $\E_{\mathcal{H}_{0},U}L_{\pi,U}^{2}(Y)\to 1$, as desired. This completes the proof of part (c).



\subsection{Proof of Lemma \ref{lem:cov}}

We first state two lemmas, which will be used in the proof of Lemma \ref{lem:cov}. The first lemma is the analogue of Lemma \ref{1.3} for $\theta>1/2$.
\begin{lem}\label{2.3}
Suppose $\theta>1/2$, and ${\bm \beta}\in [0, 2n^{-1/2}]$. Then for every positive positive integer $\ell$ we have
\begin{align}
\E_{n,\theta,{\bm \beta}}\Big(|\phi_i-t|^\ell\Big|U\Big) \le C n^{-\ell/2},
\end{align}
where $U$ is as defined in \eqref{eq:u}, and $C$ is a positive constant depending only on $\ell$ and $\theta$.
\end{lem}

For stating the second lemma, we require the following definition. Analogous to \eqref{eq:u}, define
\begin{align}\label{eq:ut}
 \widetilde{U}:=\cap_{i=1}^n \widetilde{V}_{i}\quad   \widetilde{V}_{i}:=\Big\{\phi_i\in  \Big[ 0,2\Big]\Big\}.
\end{align}
The next lemma shows that the sets $U$ and $\widetilde{U}$ occur simultaneously with high probability, and so expectations involving $U$ can be transferred to expectations involving $\widetilde{U}$ at a very low cost. This lemma will be used frequently in the rest of this section, sometimes without an explicit mention.

\begin{lem}\label{lem:cov2}
Suppose $\theta>1/2$, and ${\bm \beta}\in [0, 2n^{-1/2}]$.   Then we have the following conclusions:
\begin{enumerate}

\item[(a)]
$\log \P_{n,\theta,{\bm \beta}}(U\Delta \widetilde{U})\lesssim -n.$
\\

\item[(b)]

For any random variable $W$ such that $\E W^2\le 1$, we have
$$\Big|\E W 1\{U\}-\E W1\{\widetilde{U}\}\Big|\le \sqrt{\P_{n,\theta,{\bm \beta}}(U\Delta \widetilde{U})}.$$

\end{enumerate}
\end{lem}

The  proofs of Lemmas \ref{2.3} and \ref{lem:cov2} are deferred to section \ref{sec:appen}. We now prove a correlation bound for higher order terms, which will be used for proved Lemma \ref{lem:cov}.

\begin{lem}\label{lem:cov4}
Suppose $\theta>1/2$, and ${\bm \beta}\in [0, 2n^{-1/2}]$. Then for any pair of indices $\{i_1,i_2,i_3\}$ (not necessarily distinct), we have
$$
Cov_{n,\theta,{\bm \beta}}\Big((\phi_{i_1}-t)(\phi_{i_2}-t), \phi_{i_3}-t|\widetilde{U}\Big)\lesssim n^{-2}.
$$

\end{lem}

\begin{proof}
Setting $M(i_1,i_2,i_3):=Cov_{n,\theta,{\bm \beta}}\Big((\phi_{i_1}-t)(\phi_{i_2}-t), \phi_{i_3}-t|\widetilde{U}\Big)$, we claim that
\begin{align}\label{eq:Mclaim}
\max_{1\le i_1,i_2\le n}\Big|M(i_1,i_2,i_3)-\frac{\theta^3 \text{sech}^6(2\theta t)}{(n-1)^3} \sum_{j_1\ne i_1, j_2\ne i_2, j_3\ne i_3}\sum_{u\in (i_1,j_1), v\in (i_2, j_2), w\in (i_3, j_3)} M(u,v,w)\Big|=O(n^{-2}).
\end{align}
We first complete the proof of the lemma, deferring the proof of \eqref{eq:Mclaim}. The above display implies the existence of a constant $C$ free of $n$, such that
\begin{align}\label{eq:Mclaim2}
\max_{1\le i_1,i_2,i_3\le n}\Big|M(i_1, i_2, i_3)-\sum_{1\le j_1, j_2, j_3\le n} B_n\Big((i_1, i_2, i_3), (j_1, j_2, j_3)\Big)\Big|\le \frac{C}{n^2},
\end{align}
where $B_n$ is a symmetric $n^3\times n^3$ matrix with non-negative entries, satisfying
\begin{align*}
\sum_{1\le j_1, j_2, j_3\le n}B_n\Big(
(i_1,i_2,i_3),(j_1,j_2,j_3)\Big) = 8\theta^3 \text{sech}^6(2\theta t)<1.
\end{align*}
Thus the matrix $({\bf I}-B_n)^{-1}$ has $\ell_\infty$ operator norm  equal to $(1-8\theta^3 \text{sech}^6(2\theta t))^{-1}<\infty$, and so \eqref{eq:Mclaim2} gives
$$\max_{1\le i_1,i_2,i_3\le n} |M(i_1,i_2,i_3)|\le C(1-8\theta^3 \text{sech}^6(2\theta t))^{-1} n^{-2},$$
from which the desired conclusion follows.
\\

It thus remains to verify \eqref{eq:Mclaim}. 
There are various possibilities depending on which of the indices $\{i,j,\ell\}$ are distinct. Below we argue the case $i_1=i_2=i$ and $i_3=j$, with $\{i,j\}$ distinct, noting that the bound follows by similar calculations for other choices. To this end, setting $k_{i,t}:=k_i-(n-1)t$ note that $(\phi_i-t|Y)\sim N\Big(\frac{k_{i,t}}{n-1}, \frac{1}{(n-1)\theta}\Big)$. Consequently, we have
\begin{align}
\label{eq:cov-1}
\notag&\P_{n,\theta,{\bm \beta}}(\widetilde{U})Cov_{n,\theta,{\bm \beta}}\Big((\phi_i-t)^2, \phi_j-t|\widetilde{U}\Big)\\
\notag=&\E_{n,\theta,{\bm \beta}} \Big[(\phi_i-t)^2 (\phi_j-t) 1\{\widetilde{U}\}\Big]-\E_{n,\theta,{\bm \beta}} \Big[(\phi_i-t)^21\{\widetilde{U}\}\Big]\E_{n,\theta,{\bm \beta}} \Big[ (\phi_j-t) 1\{\widetilde{U}\}\Big]\\
\notag=&\E_{n,\theta,{\bm \beta}} \Big[(\phi_i-t)^2 (\phi_j-t) 1\{{U}\}\Big]-\E_{n,\theta,{\bm \beta}} \Big[(\phi_i-t)^21\{{U}\}\Big]\E_{n,\theta,{\bm \beta}} \Big[ (\phi_j-t) 1\{{U}\}\Big]+O(e^{-cn})\\
\notag=&\E_{n,\theta,{\bm \beta}} \left[\Big(\frac{k_{i,t}^2}{(n-1)^2}+\frac{1}{(n-1)\theta} \Big) \frac{k_{i,t}}{n-1} 1\{U\}\right]\\
\notag-&\E_{n,\theta,{\bm \beta}} \left[\Big(\frac{k_{i,t}^2}{(n-1)^2}+\frac{1}{(n-1)\theta} \Big)  1\{U\}\right] \E_{n,\theta,{\bm \beta}} \left[\frac{k_{i,t}}{n-1} 1\{U\}\right]+O(e^{-cn})\\
\notag=&\P_{n,\theta,{\bm \beta}}(U) Cov_{n,\theta,{\bm \beta}}\Big(\frac{k_{i,t}^2}{(n-1)^2}+\frac{1}{(n-1)\theta}, \frac{k_{j,t}}{n-1}|U\Big)+O(e^{-cn})\\
\notag=&\frac{1}{(n-1)^3}\P_{n,\theta,{\bm \beta}}(U) Cov_{n,\theta,{\bm \beta}}\Big(k_{i,t}^2,k_{j,t}|U\Big)+O(e^{-cn})\\
=&\frac{1}{(n-1)^3} \P_{n,\theta,{\bm \beta}}(U) \sum_{a_1, a_2 \ne i, b\ne j} Cov\Big(Y_{ia_1,t} Y_{ia_2,t}, Y_{jb,t}|U\Big)+O(e^{-cn}),
\end{align}
where $Y_{ij,t}:=Y_{ij}-t$, and the change from $U$ to $\widetilde{U}$ uses Lemma \ref{lem:cov2} and incurs the cost $O(e^{-cn})$.
Proceeding to estimate the RHS of \eqref{eq:cov-1}, 
set $r_{ij,t}:=\E(Y_{ij}|\phi)$, and for $a_1\ne a_2$ note that
\begin{align}
\label{eq:cov-2}
\notag&\P_{n,\theta,{\bm \beta}}(U)Cov_{n,\theta,{\bm \beta}}\Big(Y_{ia_1,t} Y_{ia_2,t}, Y_{jb,t}|U\Big)\\
\notag=&\E_{n,\theta,{\bm \beta}} (Y_{ia_1,t} Y_{ia_2,t} Y_{jb,t}1\{U\}) -\E_{n,\theta,{\bm \beta}}(Y_{ia_1,t}  Y_{ia_2,t}1\{U\})\E_{n,\theta,{\bm \beta}}(Y_{jb}1\{U\})\\
\notag=&\E_{n,\theta,{\bm \beta}} \Big[Y_{ia_1,t} Y_{ia_2,t} Y_{jb,t}1\{\widetilde{U}\}\Big] -\E_{n,\theta,{\bm \beta}}\Big[Y_{ia_1,t}  Y_{ia_2,t}1\{\widetilde{U}\}\Big]\E_{n,\theta,{\bm \beta}}\Big[Y_{jb,t}1\{\widetilde{U}\}\Big]+O(e^{-cn})\\
\notag=&\E_{n,\theta,{\bm \beta}} \Big[r_{ia_1,t} r_{ia_2,t} r_{jb,t}1\{\widetilde{U}\}\Big] -\E_{n,\theta,{\bm \beta}}\Big[r_{ia_1,t}  r_{ia_2,t}1\{\widetilde{U}\}\Big]\E_{n,\theta,{\bm \beta}}\Big[r_{jb}1\{\widetilde{U}\}\Big]+O(e^{-cn})\\
=&\P_{n,\theta,{\bm \beta}}(\widetilde{U}) Cov_{n,\theta,{\bm \beta}}(r_{ia_1,t} r_{ia_2,t}, r_{jb,t}|\widetilde{U})+O(e^{-cn}).
\end{align}
In the above display, we have again moved from $U$ to $\widetilde{U}$ at a cost $O(e^{-cn})$, using  Lemma \ref{lem:cov2}.
A one term Taylor's series expansion gives
\begin{align}\label{eq:rbound}
\notag r_{ij,t}=&\tanh\Big[\theta(\phi_i+\phi_i)+(\beta_i+\beta_j)\Big]-\tanh(2\theta t)\\
=&\Big[\theta(\phi_i-t+\phi_j-t) +\frac{1}{2}(\beta_i+\beta_j)\Big]\text{sech}^2(2\theta t)+\xi_{ij}, 
\end{align}
where  $$|\xi_{ij}|\lesssim (\phi_i-t)^2+(\phi_j-t)^2+\beta_i^2+\beta_j^2\lesssim (\phi_i-t)^2+(\phi_j-t)^2+n^{-1}.$$
On taking expectations, $Cov_{n,\theta,{\bm \beta}}\Big(r_{ia_1,t} r_{ia_2,t}, r_{jb,t}|\widetilde{U}\Big)$ equals
\begin{align}\label{eq:cov-3}
\theta ^3 \text{sech}^6(2\theta t)\sum_{u\in \{i,a_1\}, v\in \{i, a_2\}, w\in \{j,b\}}Cov_{n,\theta,{\bm \beta}}\Big((\phi_u-t)(\phi_v-t),(\phi_w-t)|\widetilde{U}\Big)+O(n^{-2}),
\end{align}
where we have used Lemma \ref{2.3}. On the other hand, if $a_1=a_2=a$, then using the fact that $Y_{ia,t}^2=1+t^2-2t Y_{ia}=1-t^2-2t Y_{ia,t}$ we have
\begin{align}
\label{eq:cov-1.5}
\notag&-\frac{1}{2t}\P_{n,\theta,{\bm \beta}}(U)Cov_{n,\theta,{\bm \beta}}\Big(Y_{ia,t}^2, Y_{jb,t}|U\Big)\\
\notag=& \P_{n,\theta,{\bm \beta}}(U)Cov_{n,\theta,{\bm \beta}}\Big(Y_{ia,t}, Y_{jb,t}|U\Big)\\
\notag=&\E_{n,\theta,{\bm \beta}} (Y_{ia,t} Y_{jb,t}1\{U\}) -\E_{n,\theta,{\bm \beta}}(Y_{ia,t}1\{U\})\E_{n,\theta,{\bm \beta}}(Y_{jb}1\{U\})\\
\notag=&\E_{n,\theta,{\bm \beta}} \Big[Y_{ia,t} Y_{jb,t}1\{\widetilde{U}\}\Big] -\E_{n,\theta,{\bm \beta}}\Big[Y_{ia,t} 1\{\widetilde{U}\}\Big]\E_{n,\theta,{\bm \beta}}\Big[Y_{jb,t}1\{\widetilde{U}\}\Big]+O(e^{-cn})\\
\notag=&\E_{n,\theta,{\bm \beta}} \Big[r_{ia} r_{jb}1\{\widetilde{U}\}\Big] -\E_{n,\theta,{\bm \beta}}\Big[r_{ia}1\{\widetilde{U}\}\Big]\E_{n,\theta,{\bm \beta}}\Big[r_{jb}1\{\widetilde{U}\}\Big]+O(e^{-cn})\\
=&\P_{n,\theta,{\bm \beta}}(\widetilde{U}) Cov_{n,\theta,{\bm \beta}}(r_{ia}, r_{jb}|\widetilde{U})+O(e^{-cn})=O(n^{-1}),
\end{align}
where the last equality again uses Lemma \ref{2.3}, along with \eqref{eq:rbound}.
Combining \eqref{eq:cov-1}, \eqref{eq:cov-2}, \eqref{eq:cov-3} and \eqref{eq:cov-1.5}  we have
\begin{align*}
&Cov_{n,\theta,{\bm \beta}}\Big((\phi_i-t)^2, \phi_j-t|\widetilde{U}\Big)\\
=&\frac{\theta^3 \text{sech}^6(2 \theta t)}{(n-1)^3}\sum_{a_1, a_2\ne i, b\ne j} \sum_{u\in \{i,a_1\}, v\in \{i, a_2\}, w\in \{j,b\}}Cov_{n,\theta,{\bm \beta}}\Big((\phi_u-t)(\phi_v-t),(\phi_w-t)|\widetilde{U}\Big)+O(n^{-2}),
\end{align*}
which verifies \eqref{eq:Mclaim} for the choice $\{i_1=i_2=i, i_3=j\}$. This completes the proof of the claim.
\end{proof}

\begin{proof}[Proof of Lemma \ref{lem:cov}]

We proceed via a similar argument as in the proof of Lemma \ref{lem:cov4}. Setting $M(i_1,i_2):=Cov_{n,\theta,{\bm \beta}}(k_{i_1},k_{i_2}|U)$ for $1\le  i_1,i_2\le n$, we begin by claiming
{ \begin{align}\label{eq:Mclaim3}
\max_{1\le i_1,i_2\le n}\Big|M(i_1,i_2)-\frac{1}{(n-1)^2}\sum_{j_1\ne i_1, j_2\ne i_2}\sum_{u\in \{i_1,j_1\}, v\in \{i_2,j_2\}}[C_0+C_1(\beta_u+\beta_v)\Big] M(u,v)\Big|=O(1),
\end{align}}
where
\begin{align}\label{eq:C}
C_0:=\theta^2 \text{sech}^4(2\theta t), \quad C_1:=\frac{\theta^2}{2}  \text{sech}^2(2\theta t).
\end{align}
Given \eqref{eq:Mclaim3}, and noting that $Var_{n,\theta,{\bm \beta}}(k_i|U)=O(n)$ by Lemma \ref{2.3}, we conclude
\begin{align}\label{eq:Mclaim4}
\max_{i_1\ne i_2}\Big|M(i_1,i_2)-\sum_{j_1\ne j_2}B_n\Big((i_1,i_2),(j_1,j_2)\Big)M(j_1,j_2)\Big|=O(1),
\end{align}
where $B_n$ is a symmetric $n(n-1)$ matrix with non-negative entries, satisfying
\begin{align*}
\sum_{j_1\ne j_2} B_n\Big((i_1,i_2),(j_1,j_2)\Big)\le 8(C_0+2A)\stackrel{A\to 0}{\to} 8C_0=8\theta^3\text{sech}^6(2\theta t)<1.
\end{align*}
Thus the $\ell_\infty$ operator norm of $({\bf I}-B_n)^{-1}$ converges to $(1-8\theta^3\text{sech}^6(2\theta t))^{-1}<\infty$, which along with \eqref{eq:Mclaim4} gives
\begin{align*}
\max_{i_1\ne i_2}M(i_1,i_2)=O(1),
\end{align*}
as desired.\\

 It thus remains to verify \eqref{eq:Mclaim3}. To this end, 
for any $i\ne j$, we have
\begin{align}\label{eq:cov1}
\notag&\P_{n,\theta,{\bm \beta}}(U)Cov_{n,\theta,{\bm \beta}}(k_i,k_j|U)\\
\notag=&\P_{n,\theta,{\bm \beta}}(U)\sum_{a\ne i, b\ne j} Cov_{n,\theta,{\bm \beta}}(Y_{ia}, Y_{jb}|U)\\
\notag=&\sum_{a\ne i, b\ne j}\left\{\E_{\beta} \Big[Y_{ia}Y_{jb} 1\{U\}\Big]-\E_{\beta}\Big[ Y_{ia} 1\{U\}\Big] \E_{n,\theta,{\bm \beta}}\Big[ Y_{jb} 1\{U\}\Big]\right\}\\
\notag=&\sum_{a\ne i, b\ne j}\left\{\E_{\beta}\Big[ Y_{ia}Y_{jb} 1\{\widetilde{U}\}\Big]-\E_{\beta}\Big[ Y_{ia} 1\{\widetilde{U}\}\Big] \E_{n,\theta,{\bm \beta}} \Big[Y_{jb} 1\{\widetilde{U}\}\Big]\right\}+O(e^{-cn})\\
\notag=&\sum_{a\ne i, b\ne j}\left\{\E_{\beta} \Big[r_{ia} r_{jb}1\{\widetilde{U}\}\Big]-\E_{\beta} \Big[r_{ia} 1\{\widetilde{U}\}\Big] -\E_{\beta} \Big[r_{jb} 1\{\widetilde{U}\}\Big]\right\}+O(e^{-cn})\\
=&\P_{n,\theta,{\bm \beta}}(\widetilde{U})Cov_{n,\theta,{\bm \beta}}(r_{ia}, r_{jb}|\widetilde{U})+O(e^{-cn}).
\end{align}
In the above display, $r_{ij}:=\tanh\Big[\theta (\phi_i+\phi_a)+\frac{1}{2}(\beta_i+\beta_a)\Big]$. 
A Taylor's series expansion gives
\begin{align*}
r_{ij}=&\tanh\Big[\theta (\phi_i+\phi_j)+\frac{1}{2}(\beta_i+\beta_j)\Big] \\
=&\tanh(2\theta t)+\Big[\theta(\phi_i-t+\phi_j-t)+\frac{1}{2}(\beta_i+\beta_j)\Big] \text{sech}^2(2\theta t)\\
+&\frac{1}{2}\Big[\theta(\phi_i-t+\phi_j-t)+\frac{1}{2}(\beta_i+\beta_j)\Big]^2\tanh^{''}(2\theta t) +\xi_{ij},
\end{align*}
where $$|\xi_{ij}|\lesssim |\phi_i-t|^3+|\phi_j-t|^3+|\beta_i|^3+|\beta_j|^3\lesssim  |\phi_i-t|^3+|\phi_j-t|^3+n^{-3/2}.$$
Using the above display we have
\begin{align}
\label{eq:cov3}
Cov_{n,\theta,{\bm \beta}}(r_{ia}, r_{jb}|\widetilde{U})=&\sum_{u\in \{i,a\}, v\in \{j,b\}} \Big[C_0+C_1 (\beta_u+\beta_v) \Big]Cov_{n,\theta,{\bm \beta}}(\phi_u,\phi_v|\widetilde{U})  +O(n^{-2}),
\end{align}
where the bound on the error term uses Lemma \ref{2.3} and Lemma \ref{lem:cov4}.  In the above display, the constants $C_0,C_1$ are as in \eqref{eq:C}.

 Finally, we have
\begin{align}
\label{eq:cov4}
\notag&\P_{n,\theta,{\bm \beta}}(\widetilde{U}) Cov_{n,\theta,{\bm \beta}}(\phi_u,\phi_v|\widetilde{U})\\
\notag=&\E_{n,\theta,{\bm \beta}}\Big[\phi_u \phi_v 1\{\widetilde{U}\}\Big]-\E_{n,\theta,{\bm \beta}}\Big[\phi_u 1\{\widetilde{U}\}\Big]\E_{n,\theta,{\bm \beta}}\Big[\phi_v 1\{\widetilde{U}\}\Big]\\
\notag=&\E_{n,\theta,{\bm \beta}}\Big[\phi_u \phi_v 1\{{U}\}\Big]-\E_{n,\theta,{\bm \beta}}\Big[\phi_u 1\{{U}\}\Big]\E_{n,\theta,{\bm \beta}}\Big[\phi_v 1\{{U}\}\Big]+O(e^{-cn})\\
\notag=&\frac{1}{(n-1)^2} \E_{n,\theta,{\bm \beta}}\Big[k_u k_v 1\{U\}\Big]-\frac{1}{(n-1)^2} \E_{n,\theta,{\bm \beta}}\Big[k_u  1\{U\}\Big]  \E_{n,\theta,{\bm \beta}}\Big[k_v 1\{U\}\Big]+O(e^{-cn})\\
=&\frac{\P_{n,\theta,{\bm \beta}}(U)}{(n-1)^2} Cov_{n,\theta,{\bm \beta}}(k_u,k_v|U)+O(e^{-cn}).
\end{align}
Combining \eqref{eq:cov1},  \eqref{eq:cov3} and \eqref{eq:cov4} we have
\begin{align*}
Cov_{n,\theta,{\bm \beta}}(k_i,k_j)=&\frac{1}{(n-1)^2}\sum_{a\ne i, b\ne j}\sum_{u\in \{i,a\}, v\in \{j,b\}}[C_0+C_1(\beta_u+\beta_v)] Cov_{n,\theta,{\bm \beta}}( k_u, k_v|U)+O(1),
\end{align*}
from which \eqref{eq:Mclaim3} follows. This completes the proof of the lemma.

\end{proof}

\section{Proofs of Auxiliary Variable Lemmas}\label{sec:appen}

\subsection{Proof of Proposition \ref{thm:bayes}}

The conditional distribution $(\phi|Y)$ has a density on $\R^n$ proportional to
\begin{align*}
\exp\Big[-\frac{(n-1)\theta }{2}\sum_{i=1}^n\Big(\phi_i-\frac{k_i}{n-1}\Big)^2\Big]=&\exp\Big[-\frac{(n-1)\theta}{2}\sum_{i=1}^n\phi_i^2+\theta\sum_{i=1}^n\phi_i k_i -\frac{\theta}{2(n-1)}\sum_{i=1}^nk_i^2\Big]\\
=&\exp\Big[-\frac{(n-1)\theta}{2}\sum_{i=1}^n\phi_i^2-\frac{\theta}{2(n-1)}\sum_{i=1}^nk_i^2+\theta\sum_{i<j} Y_{ij}(\phi_i+\phi_j)\Big].
\end{align*}
Since $Y$ has a p.m.f.~proportional to $\exp\Big(\frac{\theta}{2}\sum_{i=1}^nk_i^2\Big)$, the joint distribution of $(Y,\phi)$ has a density on $\{-1,1\}^{n\choose 2}\times \R^n$ proportional to
\begin{align}\label{eq:joint_bayes}
\exp\Big[-\frac{(n-1)\theta}{2}\sum_{i=1}^n\phi_i^2+\theta\sum_{i<j} Y_{ij}(\phi_i+\phi_j)\Big].
\end{align}
\begin{enumerate}
\item[(a)]
From \eqref{eq:joint_bayes}, it follows that conditional on $\phi$ the random variables $\{Y_{ij},1\le i<j\le n\}$ are mutually independent, with $Y_{ij}$ having the distribution as in part (a).

\item[(b)]
Summing over the expression in \eqref{eq:joint_bayes}, the marginal density of $\phi$ is proportional to
\begin{align*}
&\sum_{Y\in \{-1,1\}^{n\choose 2}}\exp\Big[-\frac{(n-1)\theta}{2}\sum_{i=1}^n\phi_i^2+\theta\sum_{i<j} Y_{ij}(\phi_i+\phi_j)\Big]\\
=&2^{n\choose 2}\exp\Big[-\frac{(n-1)\theta}{2}\sum_{i=1}^n\phi_i^2+\log\cosh(\theta (\phi_i+\phi_j))\Big].
\end{align*}
Since the RHS above is proportional to $f_{n,\theta,{\bm \beta}}(\phi)$, the conclusion of part (b) follows.
\end{enumerate}

\subsection{Proof of Lemma \ref{1.3}}

For proving Lemma \ref{1.3}, we need the following two lemmas.

\begin{lem}\label{1.1}
Suppose $\theta=1/2$, and ${\bm \beta}\in [0, n^{-1/2}]$. Then there exists a positive constant $M$ free of $n$, such that
\begin{align}
\log  \P_{n,\theta,{\bm \beta}}(\sum_{i=1}^{n}(\phi_{i}-\Bar{\phi})^{2}>M)\lesssim-n.
\end{align}
\end{lem}

\begin{lem}\label{1.2}
Suppose $\theta=1/2$, and ${\bm \beta}\in [0, n^{-1/2}]$. Then there exists a positive constant $M$ free of $n$, such that
\begin{align}
\log \P_{n,\theta,{\bm \beta}}(|\Bar{\phi}|\geq Mn^{-1/4})\lesssim -n.
\end{align}

\end{lem}

\subsubsection{Proof of Lemma \ref{1.1} }

To begin, use \eqref{eq:p_recall} we get the existence of a finite positive constant $C$ free of $n$ such that
\begin{align}\label{eq:1.11}
\notag&\left|\log f_{n,\theta,{\bm\beta}}-\log f_{n,\theta,{\bm 0}}(\phi)-\frac{(n-1)sA}{2}\tanh( \bar{\phi})\right|\\
\notag\le &2CnA \sum_{i=1}^s|\phi_i-\bar{\phi}|+2CsA\sum_{i=1}^n|\phi_i-\bar{\phi}|+CnsA^2\\
\notag\le& Cn\Big[\delta^2  \sum_{i=1}^s(\phi_i-\bar{\phi})^2+\frac{sA^2}{\delta^2}\Big]+Cs \Big[\delta^2\sum_{i=1}^n(\phi_i-\bar{\phi})^2+\frac{nA^2}{\delta^2}\Big]+CnsA^2\\
\le &Cn\delta^2\sum_{i=1}^n(\phi_i-\bar{\phi})^2+\frac{3CnsA^2}{\delta^2}.
\end{align}
for any $\delta\in (0,1)$, where we use the bound $2ab\le a^2+b^2$ in the third inequality. 
Also, with $q(x)=\frac{x^2}{2}-\log\cosh( x)$ as in \eqref{eq:q}, we have $q''(x)=1 - \text{sech}^2( x)\in [0,1]$, where we use the fact that $\theta=\frac{1}{2}$. A Taylor's series expansion then gives
\begin{align}\label{eq:qq}
\Big(\frac{\phi_i+\phi_j}{2}-\bar{\phi}\Big)q'(\bar{\phi})\le q\Big(\frac{\phi_i+\phi_j}{2}\Big)- q(\bar{\phi})\le \Big(\frac{\phi_i+\phi_j}{2}-\bar{\phi}\Big)q'(\bar{\phi})+\frac{1}{2} \Big(\frac{\phi_i+\phi_j}{2}-\bar{\phi}\Big)^2,
\end{align}
which on summing over $i<j$ and invoking with  \eqref{eq:fh_0}
\begin{align}\label{eq:1.12}
\frac{n(n-1)}{2}q(\bar{\phi})+\frac{n}{8}\sum_{i=1}^n(\phi_i-\bar{\phi})^2\le -\log f_{n,\theta,{\bm 0}}(\phi)\le \frac{n(n-1)}{2}q(\bar{\phi})+\frac{n}{4}\sum_{i=1}^n(\phi_i-\bar{\phi})^2.
\end{align}
This gives
\begin{align}
\label{eq:1.13}
\notag&\P_{n,\theta,{\bm \beta}}(\sum_{i=1}^n(\phi_i-\bar{\phi})^2>M)\\
\notag=&\frac{\int_{\R^n} e^{-f_{n,\theta,{\bm\beta}}(\phi)}1\{\sum_{i=1}^n(\phi_i-\bar{\phi})^2>M\}d\phi}{\int_{\R^n} e^{-f_{n,\theta,{\bm\beta}}(\phi)}d\phi}\\
\le & e^{\frac{3CnsA^2}{\delta^2}}\frac{\int_{\R^n}  \exp\Big(-\frac{n(n-1)}{2}q(\bar{\phi})-\frac{\lambda_1}{2} \sum_{i=1}^n(\phi_i-\bar{\phi})^2\Big)1\{\sum_{i=1}^n(\phi_i-\bar{\phi})^2>M\}d\phi}{\int_{\R^n}  \exp\Big(-\frac{n(n-1)}{2}q(\bar{\phi})-\frac{\lambda_2}{2} \sum_{i=1}^n(\phi_i-\bar{\phi})^2\Big)d\phi},
\end{align}
where $\lambda_1:=\frac{\theta}{2}-2C\delta^2$ and $\lambda_2:=\theta+2C\delta^2$ are positive reals, for the choice $\delta^2:=\frac{\theta}{8C}$.
Let $O_n$ be an orthogonal matrix with first row equal to $n^{-1/2}{\bf 1}$. Then, setting $\psi=O_{n}\phi$ we have 
$\psi_{1}=\sqrt{n}\Bar{\phi}$, and $\sum_{i=2}^{n}\psi_{i}^{2}=\sum_{i=1}^{n}(\phi_{i}-\Bar{\phi})^{2}$.
Using this transformation, the ratio of integrals in the RHS of \eqref{eq:1.13} equals
\begin{align*}
 \frac{\int_{\R^n}  \exp\Big(-\frac{n(n-1)}{2}q(n^{-1/2} \psi_1 )-\frac{\lambda_1}{2} \sum_{i=2}^n\psi_i^2\Big)1\{\sum_{i=2}^n\psi_i^2>M\}d\psi}{\int_{\R^n}  \exp\Big(-\frac{n(n-1)}{2}q(n^{-1/2}\psi_1)-\frac{\lambda_2}{2} \sum_{i=2}^n\psi_i^2\Big)d\psi}
 =\Big(\frac{\lambda_2}{\lambda_1}\Big)^{\frac{n-1}{2}} \P(\chi_{n-1}^2>M\lambda_1).
\end{align*}
The desired conclusion is immediate from standard tail bounds of the $\chi_{n-1}^2$ distribution.

%
%
%

%
\subsubsection{Proof of Lemma \ref{1.2}}
To begin, note that
\begin{align*}
\P_{n,\theta,{\bm \beta}}(|\Bar{\phi}|\geq Mn^{-1/4})= &\P_{n,\theta,{\bm \beta}}(|\Bar{\phi}|\geq 2)+ \P_{n,\theta,{\bm \beta}}(M n^{-1/4}\le |\Bar{\phi}|\le 2),
\end{align*}
where $\log \P_{n,\theta,{\bm \beta}}(|\Bar{\phi}|\geq 2)\lesssim -n$ using \eqref{auxil}. 
Proceeding to bound the first term in the RHS of the above display, using an argument similar to the derivation of \eqref{eq:1.13} we get
\begin{align*}
& \P_{n,\theta,{\bm \beta}}(M n^{-1/4}\le |\Bar{\phi}|\le  2)\\
\le & e^{\frac{3CnsA^2}{\delta^2}}\frac{\int_{\R^n}  \exp\Big(-\frac{n(n-1)}{2}q(\bar{\phi})-\frac{\lambda_1}{2} \sum_{i=1}^n(\phi_i-\bar{\phi})^2\Big)1\{ M n^{-1/4}\le |\bar{\phi}|\le 2\}d\phi}{\int_{\R^n}  \exp\Big(-\frac{n(n-1)}{2}q(\bar{\phi})-\frac{\lambda_2}{2} \sum_{i=1}^n(\phi_i-\bar{\phi})^2\Big)d\phi}\\
= & e^{\frac{3CnsA^2}{\delta^2}}\frac{\int_{\R^n}  \exp\Big(-\frac{n(n-1)}{2}q(n^{-1/2}\psi_1)-\frac{\lambda_1}{2} \sum_{i=2}^n\psi_i^2\Big)1\{M n^{-1/4}\le |n^{-1/2}\psi_1|\le 2\}d\psi}{\int_{\R^n}  \exp\Big(-\frac{n(n-1)}{2}q(n^{-1/2}\psi_1)-\frac{\lambda_2}{2} \sum_{i=2}^n\psi_i^2\Big)d\psi}\\
=&e^{\frac{3CnsA^2}{\delta^2}} \Big(\frac{\lambda_2}{\lambda_1}\Big)^{\frac{n-1}{2}}\frac{\int_{\R}  \exp\Big(-\frac{n(n-1)}{2}q(n^{-1/2}\psi_1)\Big)1\{M n^{-1/4}\le |n^{-1/2}\psi_1|\le 2\}d\psi_1}{\int_{\R}  \exp\Big(-\frac{n(n-1)}{2}q(n^{-1/2}\psi_1)\Big)d\psi_1}\\
\end{align*}
where we use the orthogonal transformation $\phi\mapsto \psi$ introduced in Lemma \ref{1.1} in the last step. Now the function $q(.)$ satisfies $q'(0)=q''(0)=q'''(0)=0$, and $q''''(0)>0$. Since $q(.)$ is continuous and does not vanish anywhere else on $\R$, there exists finite positive constants $c_1,c_2$ such that $c_1 x^4\le q(x)\le c_2 x^4$ for all $x\in [-2,2]$. Using this, the ratio of integrals in the above display can be bounded by
\begin{align*}
\frac{\int_{\R}  \exp\Big(-c_1'\psi_1^4\Big)1\{|\psi_1|\ge M n^{1/4} \}d\psi_1}{\int_{\R}  \exp\Big(-c_2'\psi_1^4\Big)d\psi_1}.
\end{align*}
The desired conclusion follows from the above display using Laplace method for a suitable choice of $M$.
\subsubsection{Proof of Lemma \ref{1.3}}
Without loss of generality, it suffices to work with $\phi_{1}$. For $2\le i\le n$, using \eqref{definep} we have 
\begin{align}
\label{eq:combine}
\notag p_{1i}(\phi_1,\phi_i)=&\frac{1}{8}(\phi_{1}-\phi_{i})^{2}+q\Big(\frac{\phi_{1}+\phi_{i}}{2}\Big)-\log\cosh\Big(\frac{\phi_1+\phi_i}{2}+\frac{\beta_1+\beta_i}{2}\Big)+\log\cosh\Big(\frac{\phi_1+\phi_j}{2}\Big)\\
=&\frac{1}{8}(\phi_{1}-\phi_{i})^{2}+q\Big(\frac{\phi_{1}+\phi_{i}}{2}\Big)-\frac{\beta_1+\beta_i}{2}\tanh(\theta(\phi_1+\phi_i))+O(\beta_1+\beta_i)^2.  
\end{align}
Note that $q''(x)\in [0,1]$, which along with a Taylor's series expansion around $\bar{\phi}_1:=\frac{\sum_{j=2}^n\phi_j}{n-1}$ gives 
\begin{align*}
0\le q\Big(\frac{\phi_1+\phi_i}{2}\Big)-q(\bar{\phi}_1)
-\Big(\frac{\phi_1+\phi_i}{2}-\bar{\phi}_1\Big)q'(\bar{\phi}_1)\le \frac{1}{2}\Big(\frac{\phi_1+\phi_i}{2}-\bar{\phi}_1\Big)^2.
\end{align*}
On adding over $i\in [2,n]$ and using the previous display, this gives
\begin{align}
\label{eq:qqq}
\notag&\frac{n-1}{8}(\phi_1-\bar{\phi}_1)^2+\frac{1}{8}\sum_{i=2}^n(\phi_i-\bar{\phi}_1)^2\\
\notag\le &\frac{\theta}{4}\sum_{i=2}^n(\phi_1-\phi_i)^2+\sum_{i=2}^n q\Big(\frac{\phi_1+\phi_i}{2}\Big)-(n-1)\Big[q(\bar{\phi}_1)+ \frac{1}{2} q'(\bar{\phi}_1)(\phi_1-\bar{\phi}_1)\Big]\\
\le&\frac{n-1}{4}(\phi_1-\bar{\phi}_1)^2+\frac{1}{4}\sum_{i=2}^n(\phi_i-\bar{\phi}_1)^2.
\end{align}
Another Taylor's series approximation gives
\begin{align*}
\tanh\Big(\frac{\phi_1+\phi_i}{2}\Big)=\tanh( \bar{\phi}_1)+\Big(\frac{\phi_1+\phi_i}{2}-\bar{\phi}_1\Big)\text{sech}^2( \bar{\phi}_1)+O(\phi_1+\phi_i-2\bar{\phi}_1)^2,
\end{align*}
which on summing over $i\in [2,n]$ gives 
\begin{align}\label{eq:tanhq}
\notag
&\sum_{i=2}^n \frac{\beta_1+\beta_i}{2}\tanh(\theta(\phi_1+\phi_i))\\
\notag=&\tanh(\bar{\phi}_1)\sum_{i=2}^n \frac{\beta_1+\beta_i}{2}+\frac{1}{4}\sum_{i=2}^n \beta_i(\phi_i-\bar{\phi}_1)\text{sech}^2(\bar{\phi}_1)\\
+&O\left((n-1)A |\phi_1-\bar{\phi}_1|+(n-1)A(\phi_1-\bar{\phi}_1)^2+A\sum_{i=2}^n(\phi_i-\bar{\phi_1)^2}\right).
\end{align}
Combining \eqref{eq:qqq} and \eqref{eq:tanhq} along with \eqref{eq:combine} we get the existence of a positive constant $C$ free of $n$, such that 
\begin{align}\label{eq:combine2}
\notag&\frac{1}{9}\Big[(n-1)(\phi_1-\bar{\phi}_1)^2+\sum_{i=2}^n(\phi_i-\bar{\phi}_1)^2\Big]-C(n-1)A|\phi_1-\bar{\phi}_1|- CnA^2\\
\notag\le &\sum_{i=2}^n p_{1i}(\phi_1,\phi_i)-\varphi(\phi_i, 2\le i\le n)- \frac{n-1}{2} q'(\bar{\phi}_1)(\phi_1-\bar{\phi}_1)\\
\le& \frac{1}{3}\Big[(n-1)(\phi_1-\bar{\phi}_1)^2+\sum_{i=2}^n(\phi_i-\bar{\phi}_1)^2\Big]+C(n-1)A|\phi_1-\bar{\phi}_1|+ CnA^2.
\end{align}
In \eqref{eq:combine2}, we have set
\[\varphi(\phi_i, 2\le i\le n):=(n-1)q(\bar{\phi}_1)-\tanh(\bar{\phi}_1)\sum_{i=2}^n \frac{\beta_1+\beta_i}{2}-\frac{1}{4}\sum_{i=2}^n \beta_i(\phi_i-\bar{\phi}_1)\text{sech}^2(\bar{\phi}_1),\] which
is a function which does not depend on $\phi_1$. 
Set \[D:=\left\{\sum_{i=2}^{n}(\phi_{i}-\Bar{\phi_{1}})^{2}\leq M, |\Bar{\phi_{1}}|\leq Mn^{-1/4}\right\},\] where $M$ is a constant free of $n$ such that $\log \P_{n,\theta,{\bm \beta}}(U^c)\lesssim -n$. The existence of such a constant follows from Lemmas \ref{1.1} and \ref{1.2}. 
Then we have
\begin{align*}
&\E_{n,\theta,{\bm \beta}}|\phi_{1}-\Bar{\phi_{1}}|^{l}=\E_{n,\theta,{\bm \beta}}|\phi_{1}-\Bar{\phi_{1}}|^{l}1_{D}+\E_{n,\theta,{\bm \beta}}|\phi_{1}-\Bar{\phi_{1}}|^{l}1_{D^{c}}\\
\le &\E_{n,\theta,{\bm \beta}} \left(\E_{n,\theta,{\bm \beta}}\Big(|\phi_{1}-\Bar{\phi_{1}}^{l} |\phi_i,2\le i\le n\Big)1_{D}\right)+\sqrt{\E_{n,\theta,{\bm \beta}}|\phi_{1}-\Bar{\phi_{1}}|^{2l}}\sqrt{\P_{n,\theta,{\bm \beta}}(D^{c})}.
\end{align*}
Since $\P_{n,\theta,{\bm \beta}}(D^c)$ decays exponentially, to complete the argument it suffices to show that 
\begin{align}\label{eq:suffice_again}
\sup_{(\phi_2,\ldots,\phi_n)\in D} \E_{n,\theta,{\bm \beta}}\Big(|\phi_{1}-\Bar{\phi_{1}}|^{l} |\phi_i,2\le i\le n\Big)\lesssim n^{-\ell/2}.
\end{align}
Proceeding to show \eqref{eq:suffice_again}, using \eqref{eq:combine2}  we have 
\begin{align}
\notag
&\E_{n,\theta,{\bm \beta}}\Big(|\phi_{1}-\Bar{\phi_{1}}|^{l}\ |\phi_i, 2\le i\le n\Big)\\
\notag
=&\frac{\int_{\R}|\phi_{1}-\Bar{\phi_{1}}|^{l}\exp\Big(-\sum_{i=2}^{n}p_{1i}(\phi_1,\phi_i)\Big)d\phi_{1}}{\int_{\R}\exp\Big(-\sum_{i=2}^{n}p_{1i}(\phi_1,\phi_i)\Big)d\phi_{1}} \\
\notag
\le & \exp\left(\Big(\frac{1}{3}-\frac{1}{9}\Big)\sum\limits_{i=2}^{n}(\phi_{i}-\Bar{\phi_{1}})^{2}+2CnA^2\right) \\
\notag
\times &
\frac{\int_\R |\phi_1-\bar{\phi}_1|^\ell \exp\Big(-\frac{n-1}{3}(\phi_1-\bar{\phi}_1)^2-\frac{n-1}{2}(\phi_1-\bar{\phi}_1)q'(\bar{\phi}_1)+C\sqrt{n-1}|\phi_1-\bar{\phi}_1|\Big)d\phi_1}{\int_\R \exp\Big(-\frac{n-1}{9}(\phi_1-\bar{\phi}_1)^2-\frac{n-1}{2}(\phi_1-\bar{\phi}_1)q'(\bar{\phi}_1)-C\sqrt{n-1}|\phi_1-\bar{\phi}_1|\Big)d\phi_1}\\
\notag
\le & \exp\Big(\frac{2M}{9}+2C \Big)\\
\label{eq:combine3}
\times &
\frac{\int_\R |\phi_1-\bar{\phi}_1|^\ell \exp\Big(-\frac{n-1}{3}(\phi_1-\bar{\phi}_1)^+\frac{cM^3(n-1)}{2n^{3/4}}|\phi_1-\bar{\phi}_1|+C\sqrt{n-1}|\phi_1-\bar{\phi}_1|\Big)d\phi_1}{\int_\R \exp\Big(-\frac{n-1}{9}(\phi_1-\bar{\phi}_1)^2-\frac{cM^3(n-1)}{2 n^{3/4}}|\phi_1-\bar{\phi}_1|-C\sqrt{n-1}|\phi_1-\bar{\phi}_1|\Big)d\phi_1},   
\end{align}
where $c:=\sup_{x\in [-1,1]}\frac{|q'(x)|}{|x|^3}$. By a change of variable, the RHS of \eqref{eq:combine3} becomes
$$\frac{e^{\frac{2M}{9}+2C} }{(n-1)^{\ell/2}}
\frac{\int_\R |x|^\ell e^{-\frac{x^2}{3}+\frac{cM^3\sqrt{n-1}}{2n^{3/4}}|x|+C|x|}dx}{\int_\R e^{-\frac{x^2}{9}-\frac{cM^3\sqrt{n-1}}{2 n^{3/4}}-C|x|}dx}\lesssim n^{-\ell/2},$$
from which \eqref{eq:suffice_again} follows. This completes the proof of the lemma.

\subsection{Proof of Lemma \ref{2.3} }

We begin by proving two lemmas, which will be useful for the proof.

\begin{lem}\label{lem:calculus2}

For every $a\in \R$, define the function $q_a:\R\mapsto \R$ by setting $$q_a(x):=\theta x^2-\log\cosh(2\theta x+a).$$
Denote by $t(a)$ the largest root of the equation $q_a'(x)=0$. Then the following conclusions hold:

\begin{enumerate}
\item[(a)]
The map $t(.)$ is well defined and $\mathcal{C}_1$ on $(-2\delta,2\delta)$, for some $\delta>0$.

\item[(b)]
If $a\ge 0$, then $t(a)$ is the unique global maximizer of $q_a(.)$ in $[0,\infty)$.

\item[(c)]

There exists finite positive reals $\lambda_1',\lambda_2'$ such that for all $x,y\in [0,2]$ and $a\in [0,\delta]$ we have
$$\lambda_1'[(x-t)^2+(y-t)^2]-\lambda_1'a^2\le q_a(x)-q_a(t(a))\le \lambda_2'[(x-t)^2+(y-t)^2]+\lambda_2'a^2.$$
\end{enumerate}
\end{lem}

\begin{lem}\label{2.1}
Suppose $\theta>1/2$, and ${\bm \beta}\in [0, 2n^{-1/2}]^n$. Then there exists a positive constant $M$ depending on $n$ such that
\begin{align}
\log \P_{n,\theta,{\bm \beta}}\Big(\sum_{i=1}^{n}(\phi_{i}-t)^{2}>M\Big|\phi\in [0,2]^n\Big)\lesssim -n.
\end{align}
\end{lem}

\subsubsection{Proof of Lemma \ref{lem:calculus2}}

\begin{enumerate}
\item[(a)]
Since $q_a'(x)=2\theta [x-\tanh(2\theta x+a)]$, it follows that $q_a(.)$ has an odd number of roots in $\R$, and so the maximum root is well defined for all $a\in \R$, and satisfies $x=\tanh(2\theta x+a)$. If $a=0$, then the desired conclusion follows from Lemma \ref{lem:calculus}, with $t(0)=t$. Since $q_0'''(t)\ne 0$, we must have $q_0''(t)>0$, and so using Implicit function theorem, there exists $\delta>0$ (depending on $\theta$) such that for all $a\in (-2\delta,2\delta)$ the map $t(.)$ is $\mathcal{C}_1$. 
\\

\item[(b)]

If $a=0$ then the conclusion follows from Lemma \ref{lem:calculus}, and so we assume $a>0$.
Since $q_a(x)\to \infty$ as $x\to\infty$, the function $q_a(.)$ attains a global minima at a finite number in $[0,\infty)$. Also since
$q_a'(0+)=-2\theta\tanh(2\theta+a)<0$,  $0$ is not a minima of $q_a(.)$. Since $q_a(.)$ has a unique root in $(0,\infty)$ for $a>0$, the desired conclusion follows. 
\\

\item[(c)]
Define the function $Q(.,.): [0,\delta]\times [0,2]\mapsto \R$ by setting
 \begin{align*}
 Q(a,t):=&\frac{q_a(x)-q_a(t(a))}{(x-t((a))^2}\text{ if }x\ne t(a),\\
 =&\frac{1}{2}q_a''(t(a))\text{ if }x=t(a).
 \end{align*}
Using part (b) we have $Q(.,.)$ is strictly positive point-wise, as $t(a)$ is the unique global minimizer of $q_a(.)$ in $[0,\infty)$. On the other hand, 
using part (a) we have $Q(.,.)$ is continuous. Since a continuous function on a compact set attains its maximum and minimum, we have
\begin{align}\label{eq:'}
\lambda_1'':=\inf_{a\in [0,\delta], x\in [0,2]}Q(a,t)\le \sup_{a\in [0,\delta], x\in [0,2]}Q(a,t)=:\lambda_2'',
\end{align}
Using \eqref{eq:'} we get
\begin{align*}
q_a(x)-q_a(t(a))\le \lambda_2''[(x-t(a))^2+(y-t(a))^2]
\le  \lambda_2''[(x-t)^2+(y-t)^2]+4\lambda_2''(t(a)-t)^2,\\
q_a(x)-q_a(t(a))\ge \lambda_1''[(x-t(a))^2+(y-t(a))^2]\ge \lambda_1''[(x-t)^2+(y-t)^2]-4\lambda_1''(t(a)-t)^2.
\end{align*}
The desired conclusion then follows on using part (a) to note the existence of $c>0$ such that $|t(a)-t|\le c|a|$ for all $a\in [-\delta,\delta]$. 

\end{enumerate}


\subsubsection{Proof of Lemma \ref{2.1}}

With $p_{ij}(\phi_i,\phi_j)$ as in \eqref{definep} we can write
\begin{align}\label{definep2}
p_{ij}(\phi_i,\phi_j)=\frac{\theta}{4}(\phi_i-\phi_j)^2+q_{\frac{\beta_i+\beta_j}{2}}\Big(\frac{\phi_i+\phi_j}{2}\Big),
\end{align}
where the function $q_a(.)$ is defined in Lemma \ref{lem:calculus2}. 
For $\phi_i, \phi_j\in [0,2]$ and $\beta_i,\beta_j\in [0,2n^{-1/2}]$, using part (c) of Lemma \ref{lem:calculus2} gives
\begin{align*}
\lambda_1'\Big[\frac{\phi_i+\phi_j}{2}-t\Big]^2-\frac{\lambda_1'}{n}\le q_{\frac{\beta_i+\beta_j}{2}}\Big(\frac{\phi_i+\phi_j}{2}\Big)-q_{\frac{\beta_i+\beta_j}{2}}\left(t\Big(\frac{\beta_i+\beta_j}{2}\Big)\right)\le \lambda_2'\Big[\frac{\phi_i+\phi_j}{2}-t\Big]^2+\frac{\lambda_2'}{n}.
\end{align*}
 Using this along with \eqref{definep2}, this gives the existence of finite positive constants $\lambda_1$ and $\lambda_2$, such that 
\begin{align}
\begin{split}\label{definep3}
  \frac{\lambda_1}{2}\Big[(\phi_i-t)^2+(\phi_j-t)^2\Big]-n\lambda_1\le &p_{ij}(\phi_i,\phi_j)-q_{\frac{\beta_i+\beta_j}{2}}\left(t\Big(\frac{\beta_i+\beta_j}{2}\Big)\right)\\
  \le &  \frac{\lambda_2}{2}\Big[(\phi_i-t)^2+(\phi_j-t)^2\Big]+n\lambda_2.
  \end{split}
\end{align}
 Summing over $i<j$ we get
\begin{align*}
\frac{(n-1)\lambda_1}{2}\sum_{i=1}^n(\phi_i-t)^2-n\lambda_1\le &\sum_{i<j}p_{ij}(\phi_i,\phi_j)-\sum_{i<j} q_{\frac{\beta_i+\beta_j}{2}}\left(t\Big(\frac{\beta_i+\beta_j}{2}\Big)\right)\\
\le& \frac{(n-1)\lambda_2}{2}\sum_{i=1}^n(\phi_i-t)^2+n\lambda_2,
\end{align*}
which gives
\begin{align*}
\P_{n,\theta,{\bm \beta}}\Big(\sum_{i=1}^n(\phi_i-t)^2>M\Big|\phi\in [0,2]^n\Big)=&\frac{\int_{\R^n} e^{-\sum_{i<j} p(\phi_i,\phi_i)} 1\{\sum_{i=1}^n(\phi_i-t)^2>M\} d\phi}{\int_{\R^n} e^{-\sum_{i<j} p(\phi_i,\phi_i)}d\phi_1}\\
\le &e^{n(\lambda_1+\lambda_2)}\frac{\int_{\R^n} e^{-\frac{(n-1)\lambda_1 }{2}\sum_{i=1}^n(\phi_i-t)^2} 1\{\sum_{i=1}^n(\phi_i-t)^2>M\} d\phi}{\int_{[0,2]^n} e^{-\frac{(n-1)\lambda_2 }{2}\sum_{i=1}^n(\phi_i-t)^2}  d\phi}\\
\le &e^{n(\lambda_1+\lambda_2)}\Big(\frac{\lambda_2}{\lambda_1}\Big)^{n/2}\frac{\P(\chi_n^2>(n-1)M\lambda_1)}{\P\Big(N(0,\lambda_2^{-1})\in \sqrt{n-1}[-t, 2-t]\Big)^n},
\end{align*}
The desired conclusion then follows on using standard tail bounds of the $\chi_n^2$ distribution.

\subsubsection{Proof of Lemma \ref{2.3}}

 Summing over display \eqref{definep3} for $j\in [2,n]$ gives
\begin{align*}
\frac{\lambda_1}{2}\sum_{j=2}^n(\phi_j-t)^2-\frac{\lambda_1}{2}(\phi_1-t)^2-\lambda_1\le &\sum_{j=2}^n\Big[p_{1j}(\phi_1,\phi_j)-q_{\frac{\beta_1+\beta_j}{2}}(t_1, t_j)\Big]\\
\le &\frac{\lambda_2}{2}\sum_{j=2}^n(\phi_j-t)^2+\frac{(n-1)\lambda_2}{2}(\phi_1-t)^2+\lambda_2.
\end{align*}
Thus, with $D:=\{\sum_{j=2}^n(\phi_j-t)^2\le M\}$, for $(\phi_2,\ldots,\phi_n)\in D\cap [0,2]^{n-1}$ we have
\begin{align*}
&\E_{n,\theta,\bm{\beta}}\Big(|\phi_1-t|^\ell 1\{\phi_1\in [0,2]\} \Big|\phi_i,i\neq 1\Big)\\
 =&\frac{\int_{[0,2]}|\phi_1-t|^\ell \prod\limits_{i=2}^ne^{-p_{1i}(\phi_1,\phi_i)}d\phi_1}{\int_{[0,2]}\prod\limits_{i=2}^ne^{-p_{1i}(\phi_1,\phi_i))}d\phi_1}\\
  \leq &e^{\frac{\lambda_2-\lambda_1}{2}\sum\limits_{i=2}^n(\phi_i-t)^2+\lambda_1+\lambda_2}\frac{\int_{[0,2]} \exp\Big(\frac{-(n-1)\lambda_1 (\phi_1-t)^2}{2}\Big)|\phi_1-t|^\ell d\phi_1}{\int_{[0,2]}  \exp\Big(\frac{-(n-1)\lambda_1 (\phi_1-t)^2}{2}\Big) d\phi_1}\\
  \le &e^{(\lambda_2-\lambda_1)M+\lambda_1+\lambda_2}\sqrt{\frac{\lambda_2}{(n-1)^\ell \lambda_1}}
  \frac{\E |N(0,\lambda_1^{-1})|^\ell}{\P(N\Big(0,\lambda_2^{-1})\in [-t\sqrt{n-1}, (2-t)\sqrt{n-1}]\Big)}.
   \end{align*}
   Since the probability in the denominator above converges to $1$, using tail estimates of the normal distribution we get
   $$\sup_{(\phi_2,\ldots,\phi_n)\in \{ D\cap [0,2]^{n-1}\}} \E\Big[|\phi_1-t|^\ell  1\{\phi_1\in [0,2]\}\Big|\phi_i, 2\le i\le n\Big]\lesssim n^{-\ell/2}.$$ The desired conclusion follows from the last display above, and using Lemma \ref{2.1} to get that $\log \P_{n,\theta,{\bm \beta}}(D^c|\phi\in [0,2]^n)\lesssim -n$.

\subsection{Proof of Lemma \ref{lem:cov2}}

We first prove the following lemma which will be used in proving Lemma \ref{lem:cov2}.

\begin{lem}\label{lem:cov22}
Suppose $\theta>1/2$, and ${\bm \beta}\in [0,2n^{-1/2}]^n$. Then for any $\delta>0$  there exists a constant $c$ such that
$$\log \P_{n,\theta,{\bm \beta}}(\max_{i\in [n]}|\phi_i-t|>\delta|\phi\in [0,2]^n)\le -cn.$$
\end{lem}

\subsubsection{Proof of Lemma \ref{lem:cov22}}
Set $D:=\{\sum_{i=2}^n(\phi_i-t)^2\le M\}$, and 
use \eqref{definep3} to note that for any $\delta>0$ and $(\phi_2,\ldots,\phi_n)\in D\cap [0,2]^{n-1}$ we have
\begin{align*}
&\P_{n,\theta,{\bm \beta}}(|\phi_1-t|>\delta,\phi_1\in [0,2]|\phi_i,2\le i\le n) \\
=&\frac{\int_{[0,2]}1\{|\phi_i-t|>\delta\}\prod\limits_{i=2}^ne^{-p_{1i}(\phi_1,\phi_i)}d\phi_1}{\int_{[0,2]}\prod\limits_{i=2}^ne^{-p_{1i}(\phi_1,\phi_i)}d\phi_1}\\
  \leq &e^{\frac{\lambda_2-\lambda_1}{2}\sum\limits_{i=2}^n(\phi_i-t)^2+\lambda_1+\lambda_2}\frac{\int_{[0,2]} \exp\Big(\frac{-(n-1)\lambda_1 (\phi_1-t)^2}{2}\Big)1\{|\phi_i-t|>\delta\} d\phi_1}{\int_{[0,2]}  \exp\Big(\frac{-(n-1)\lambda_1 (\phi_1-t)^2}{2} \Big)d\phi_1}\\
  \le &e^{(\lambda_2-\lambda_1)M+\lambda_1+\lambda_2}\sqrt{\frac{\lambda_2}{\lambda_1}} 
  \frac{\P\Big(N(0,\lambda_1^{-1})|>\delta\sqrt{n-1}\Big)}{\P(N\Big(0,\lambda_2^{-1})\in [-t\sqrt{n-1}, 2\sqrt{n-1}]\Big)}.
   \end{align*}
   From the above display, we get
   $$\sup_{(\phi_2,\ldots,\phi_n)\in D\cap [0,2]^{n-1}}  \log \P(|\phi_1-t|>\delta 1\{\phi_1\in [0,2]|\phi_i,2\le i\le n)\lesssim -n.$$
 Recalling that $\log\P_{n,\theta,{\bm \beta}}(D^c|\phi\in [0,2]^n)\lesssim -n$, we get $$\log \P_{n,\theta,{\bm \beta}}\Big(|\phi_1-t|>\delta,\Big|\phi\in [0,2]^n\Big)\lesssim -n.$$ A similar argument applies to all co-ordinates of $\phi$, and so a union bound gives
 $$\log \P_{n,\theta,{\bm \beta}}\Big(\max_{i\in [n]}|\phi_i-t|>\delta\Big|\phi\in [0,2]^n\Big)\lesssim -n,$$
 as desired.

\subsubsection{Proof of Lemma \ref{lem:cov2}}
\begin{enumerate}
\item[(a)]
Note that
 \begin{align*}
 \P_{n,\theta,{\bm \beta}}( U\cap \widetilde{U}^c)=& \P_{n,\theta,{\bm \beta}}(\min_{i\in [n]}\phi_i\ge 0, \max_{i\in [n]}k_i<\frac{(n-1)t}{2}\Big)\\
 \le & \P_{n,\theta,{\bm \beta}}\Big(\min_{i\in [n]}\phi_i\ge \frac{3t}{4}, \max_{i\in [n]}k_i<\frac{(n-1)t}{2}\Big)+ \P_{n,\theta,{\bm \beta}}\Big(\min_{i\in [n]}\phi_i\le \frac{3t}{4}|\phi \in [0,2]^n\Big).
 \end{align*}
 The two terms in the RHS above decays exponentially using \eqref{auxil} and Lemma \ref{lem:cov22} respectively, and so
\begin{align*}
\log \P_{n,\theta,{\bm \beta}}( U\cap \widetilde{U}^c)\lesssim -n.
\end{align*}
Also,
\begin{align*}
\log\P_{n,\theta,{\bm \beta}}( U^c\cap \widetilde{U})\lesssim -n
\end{align*}
using \eqref{auxil}. The desired conclusion follows on combining the last two displays.
\\
 
\item[(b)]
This follows on using Cauchy-Schwarz inequality to note that
\begin{align*}
\E_{n,\theta,{\bm \beta}}(W1\{U\}-W1\{\widetilde{U}\})\le &\sqrt{\E_{n,\theta,{\bm \beta}}W^2} \sqrt{ \P_{n,\theta,{\bm \beta}}(U)+\P_{n,\theta,{\bm \beta}}(\widetilde{U})-2\P_{n,\theta,{\bm \beta}}(U\cap \widetilde{U})}\\
\le &\sqrt{\P_{n,\theta,{\bm \beta}}(U\Delta \widetilde{U})}.
\end{align*}

\end{enumerate}


\section*{Acknowledgements}

The authors thank Rajarshi Mukherjee for helpful discussions throughout the project. SM gratefully thanks NSF (DMS 1712037) for support during this research.

\renewcommand\refname{Reference}
\bibliographystyle{plain}
\bibstyle{unsrt}
\bibliography{main}
\end{document}